\documentclass{amsart}

\usepackage{amsmath}
\usepackage{amsfonts}
\usepackage{amssymb}
\usepackage{amsthm}
\usepackage{array}
\usepackage{bbm}
\usepackage{chngpage}
\usepackage{comment}
\usepackage{float}
\usepackage[OT2,T1]{fontenc}
\usepackage{graphicx,caption}
\usepackage[export]{adjustbox}
\usepackage{longtable}
\usepackage{mathrsfs}
\usepackage{mathtools}
\usepackage{wrapfig}
\usepackage{cutwin}
\usepackage{shapepar}
\usepackage{tikz}
\usepackage[lowtilde]{url}
\usepackage{subcaption}
\usetikzlibrary{matrix,arrows}
\usepackage{tabularx}
\usepackage{multirow}
\definecolor{citeclr}{rgb}{0.55, 0.55, 0.64}
\definecolor{linkclr}{rgb}{0, 0.21, 0.9447}
\usepackage[colorlinks,
  linkcolor=linkclr,
  citecolor=citeclr,
  urlcolor=blue!46!cyan,
  ]{hyperref}
\usepackage[nameinlink, nosort]{cleveref}
\usepackage{enumitem}
\usepackage{diagbox}
\usepackage{etoolbox}
\usepackage{algpseudocode}
\usepackage{etoolbox}

\usepackage[margin=0.9in]{geometry}

\patchcmd{\subsection}{-.5em}{.5em}{}{}

\patchcmd{\section}{\normalfont}{\normalfont\Large}{}{}

\newtheorem{theorem}{Theorem}[section]
\newtheorem{lemma}[theorem]{Lemma}

\newtheorem{proposition}[theorem]{Proposition}

\theoremstyle{definition}
\newtheorem{definition}[theorem]{Definition}
\newtheorem{remark}[theorem]{Remark}
\newtheorem{figurecap}[theorem]{Figure}
\crefname{figurecap}{figure}{figures}
\newtheorem{tablecap}[theorem]{Table}
\crefname{tablecap}{table}{tables}

\crefname{idea}{idea}{ideas}

\crefname{observation}{observation}{observations}
\newtheorem{example}[theorem]{Example}
\AtBeginEnvironment{example}{%
  \pushQED{\qed}%
}
\AtEndEnvironment{example}{\popQED\endexample}

\newcommand{\A}{\mathbb{A}}
\newcommand{\CC}{\mathbb{C}}

\newcommand{\Q}{\mathbb{Q}}
\newcommand{\R}{\mathbb{R}}
\newcommand{\Z}{\mathbb{Z}}

\newcommand{\cF}{\mathcal{F}}

\newcommand{\cS}{\mathcal{S}}

\DeclareSymbolFont{cyrletters}{OT2}{wncyr}{m}{n}
\DeclareMathSymbol{\sha}{\mathalpha}{cyrletters}{"58}

\newcommand{\eps}{\varepsilon}

\renewcommand{\Re}{\mathrm{Re}}

\newlength{\strutheight}
\newcommand{\half}{\frac{1}{2}}

\newcommand{\third}{\frac{1}{3}}

\renewcommand{\mod}{\mspace{4mu}\mathrm{mod}\mspace{4mu}}
\newcommand*{\onesymb}{\text{\large\usefont{U}{bbold}{m}{n}1}} 
\newcommand{\one}[1]{\raisebox{-0.33pt}{\onesymb}\mspace{-1.5mu}\{#1\}}
\newcommand{\onelr}[1]{\raisebox{-0.33pt}{\onesymb}\mspace{-4.5mu}\left\{#1\right\}}

\newcommand{\cFQ}{\cF}
\newcommand{\cFY}{\cF_6}
\newcommand{\rhoY}{\rho_6}
\newcommand{\SY}{S^{(6)}}
\newcommand{\m}{n}
\newcommand{\M}{y}
\newcommand{\x}{X}
\newcommand{\q}{p_*}
\newcommand{\QD}{q}
\newcommand{\QM}{m}
\newcommand{\T}{T}
\newcommand{\QT}{\QM_\T}

\author{Alex Cowan}
\address{Department of Mathematics, University of Waterloo, Waterloo, ON, Canada}
\email{alex.cowan@uwaterloo.ca}

\title{Conductor distributions of elliptic curves}
\date{\today}

\AtBeginDocument{%
   \def\MR#1{}
}

\begin{document}
\begin{abstract}
  We determine the distribution of the conductors $N$ of rational elliptic curves when ordered by naive height $H$,
  in the form of an explicit density function for the ratios $N/H$.
  Our work is essentially an effective version of the Brumer--McGuinness--Watkins heuristic.
  
  Applying our results to the problem of enumerating elliptic curves by conductor
  gives the strongest bounds yet for the number of elliptic curves which have conductor much smaller than their height for ranges up to $H \ll N^{1.2165}$.
\end{abstract}
\maketitle
\tableofcontents

\section{Introduction}\label{sec:intro}

Elliptic curves are most naturally ordered by conductor but most easily ordered by height. Converting between these two orderings is an interesting and difficult problem \cite{watkins, SSW, cremona_sadek}.
We study this problem for the following two families $\cFQ$ and $\cFY$ of elliptic curves $E_{a,b}: y^2 = x^3 + ax + b$:
\begin{definition}\label{cF_def}
  For $H > 0$, define
\begin{align*}
  \cFQ(H) \coloneqq \left\{E_{a,b} \,:\, |a| \leqslant H^{\frac{1}{3}},\, |b| \leqslant H^{\half},\, p^4 \mid a \Rightarrow p^6 \nmid b\right\}.
\end{align*}
\end{definition}

\begin{definition}\label{cF_tilde_def}
  For $H > 0$ and $r,t \in \Z$ such that $3\nmid r$ and $2\nmid t$, define
\begin{align*}
  \tilde{\cF}(H) \coloneqq \left\{E_{a,b} \,:\, a = r\mod 6,\, b = t\mod 6,\, |a| \leqslant H^{\frac{1}{3}},\, |b| \leqslant H^{\half},\, p^4 \mid a \Rightarrow p^6 \nmid b\right\}.
\end{align*}
\end{definition}

Every $\Q$-isomorphism class of elliptic curves over $\Q$ has a unique representative in $\cFQ \coloneqq \bigcup_H \cFQ(H)$. As the conductor is an isomorphism invariant, our results for $\cFQ$ apply to the set of all rational elliptic curve isomorphism classes ordered by minimal height among integral short Weierstrass models. The family $\cFY$ appears in \cite{young:elliptic_curves,DHP}, and our results here are needed in \cite{ratiosconjecture}, which relies on \cite{DHP}.

Relationships between an elliptic curve's height $H = H_E$ and its conductor $N = N_E$ are of great interest, but very elusive. Of particular interest is their relative sizes: how much smaller than $H$ can $N$ be? The crux of this question is to do with exceptional cancellation.
The discriminant $\Delta \coloneqq -16(4a^3 + 27b^2)$ being much smaller than the height $H$
reflects interplay between the additive and multiplicative structures of the integers that produces extreme behaviour. The connection between $N$ and $H$ is the adelic version of the matter, with $\Delta$ vs.\ $H$ being the manifestation at the infinite place.

Very little is known about this difficult and fundamental question. It is expected, based on probabilistic heuristics, that significant cancellation happens rarely. For instance, $\#\cF(H) \asymp H^{\frac{5}{6}}$ and it is expected that the number of elliptic curves with conductor at most $H$ is also $\asymp H^{\frac{5}{6}}$ \cite{bm,watkins}. However, the best result currently is that this count is $\ll H^{1 + \eps}$ \cite{duke_kowalski}, i.e.\ even the very coarse question of counting by conductor is far out of reach. Answers to any questions finer than this one regarding the zeroth moment have not been proposed.

Our first result, \cref{thm:conductor_distribution_FN}, gives an explicit description of the distribution of relative sizes of $H$ and $N$ for $100\%$ of the elliptic curves in $\cF$ or $\cFY$. I.e., if one were to pick $E \in \cF(H)$ uniformly at random, what is the distribution (i.e.\ probability measure) of the quantity $N_E/H$? A plot of this distribution, which is supported on the interval $(0, 496]$, 
is shown in \cref{fig:thm1} below. The structure one discovers is quite rich: the distribution converges weakly to a continuous function which is not differentiable at infinitely many points.
  
To state \cref{thm:conductor_distribution_FN}'s formula for this continuous function we must introduce some notation. 

Let $\zeta^{(\m)}(s)$ denote the Riemann zeta function with the Euler product factors at primes dividing $\m$ removed: 
\begin{align}
  \nonumber
  \zeta^{(\m)}(s) \coloneqq \zeta(s) \prod_{p\mid \m} (1 - p^{-s}).
\end{align}

Let $F_\Delta$ denote the following approximation of $\#\{E \in \cF(H) \,:\, \Delta_E < \lambda H\} / \#\cF(H)$ (or idem 
with $\cFQ \mapsto \cFY$):
\begin{definition}\label{FDelta_def}
  \begin{align*}
    F_\Delta(\lambda) \coloneqq \frac{1}{4}\int_{-1}^1\int_{-1}^1 \begin{cases}1 & \text{if $-16(4\alpha^3 + 27\beta^2) < \lambda$} \\ 0 & \text{otherwise}\end{cases} \,d\alpha \,d\beta.
  \end{align*}
\end{definition}

Next, let $\rho(p,\m)$ denote the natural density of $(a,b) \in \Z^2$
such that $E = E_{a,b}$'s discriminant to conductor ratio has $p$-part equal to the $p$-part of $\m$, i.e.:
\begin{definition}\label{rho_def}
  For any prime $p$ and integer $\m$,
  \begin{align*}
    \rho(p,\m) \coloneqq \lim_{H\to\infty} \frac{1}{4H^{\frac{5}{6}}} \#\left\{(a,b) \in \Z^2 \,:\, |a|^3, |b|^2 < H,\,\, 4a^3 + 27b^2 \neq 0,\,\, \gcd\!\left(\frac{\Delta_{E}}{N_{E}}, p^\infty\right) = \gcd(\m, p^\infty)\right\}.
  \end{align*}
\end{definition}
The values of $\rho(p,\m)$ are the simple rational functions of $p$ and $\gcd(\m, p^\infty)$ tabulated below. 
\begin{table}[H]
  \begin{tabular}{|l||l|l|l|}
    \hline
    $\gcd(\m, p^\infty)$ & $p \geqslant 5$ & $p = 2$ & $p = 3$ \rule{0pt}{1.4em}\\
        [1.1ex]
        \hline
        \hline
        $p^0$ & $1 - \frac{1}{p^2}$                                                       & $2^{-1}$ & $8 \cdot 3^{-2}$ \rule{0pt}{1.2em}\\ [0.7ex]
        $p^1$ & $p^{-2}\big(1-\frac{1}{p}\big)$                               & $2^{-2}$ & $2 \cdot 3^{-3}$ \\ [0.7ex]
        $p^2$ & $p^{-3}\big(1-\frac{1}{p}\big)$                               & $2^{-3}$ & $2 \cdot 3^{-4}$ \\ [0.7ex]
        $p^3$ & $p^{-4}\big(1-\frac{1}{p}\big)\big(1 - \frac{1}{p}\big)$ & $0$ & $0$ \\ [0.7ex]
        $p^4$ & $p^{-5}\big(1-\frac{1}{p}\big)\big(2 - \frac{1}{p}\big)$ & $2^{-4}$ & $2 \cdot 3^{-5}$ \\ [0.7ex]
        $p^5$ & $p^{-6}\big(1-\frac{1}{p}\big)\big(2 - \frac{2}{p}\big)$ & $2^{-6}$ & $2 \cdot 3^{-7}$ \\ [0.7ex]
        $p^6$ & $p^{-7}\big(1-\frac{1}{p}\big)\big(3 - \frac{2}{p}\big)$ & $3 \cdot 2^{-7}$ & $10 \cdot 3^{-8}$ \\ [0.7ex]
        $p^7$ & $p^{-8}\big(1-\frac{1}{p}\big)\big(3 - \frac{2}{p}\big)$ & $3 \cdot 2^{-8}$ & $10 \cdot 3^{-9}$ \\ [0.7ex]
        $p^8$ & $p^{-9}\big(1-\frac{1}{p}\big)\big(3 - \frac{2}{p}\big)$ & $3 \cdot 2^{-9}$ & $10 \cdot 3^{-10}$ \\ [0.7ex]
        $p^9$ & $p^{-10}\big(1-\frac{1}{p}\big)\big(2 - \frac{2}{p}\big)$ & $2^{-10}$ & $4 \cdot 3^{-11}$ \\ [0.7ex]
        $p^{10}$ & $p^{-11}\big(1-\frac{1}{p}\big)\big(2 - \frac{2}{p}\big)$ & $2^{-11}$ & $4 \cdot 3^{-12}$ \\ [0.7ex]
        $p^{11}$ & $p^{-12}\big(1-\frac{1}{p}\big)\big(2 - \frac{2}{p}\big)$ & $2^{-12}$ & $4 \cdot 3^{-13}$ \\ [0.7ex]
        $p^{12}$ & $p^{-13}\big(1-\frac{1}{p}\big)\big(2 - \frac{2}{p}\big)$ & $21 \cdot 2^{-13}$ & $148 \cdot 3^{-14}$ \\ [0.7ex]
        $p^k$, $k \geqslant 13$ & $p^{-k-1}\big(1-\frac{1}{p}\big)\big(2 - \frac{2}{p}\big)$ & $5 \cdot 2^{-k-1}$ & $40 \cdot 3^{-k-2}$ \\ [0.9ex]
        \hline
  \end{tabular}
  \begin{tablecap}\label{table:rho_def}
    Values of $\rho(p,\m)$.
  \end{tablecap}
  \vspace{-\baselineskip}
\end{table}  

The function $\rho(p,\m)$ will be used when working with $\cFQ$. The appropriate analogue $\rhoY(p,\m)$ for use with $\cFY$ is the following.

\begin{definition}\label{rho_tilde_def}
  Define
  \begin{align*}
    \rhoY(2,\m) \coloneqq
    \begin{cases}
      \half & \gcd(\m, 2^\infty) = 1\\
      \frac{1}{4} & \gcd(\m, 2^\infty) = \text{$2$ or $4$}\\
      0 & \text{otherwise,}
    \end{cases}
    \quad\quad\quad
    \rhoY(3,\m) \coloneqq
    \begin{cases}
      1 & \gcd(\m, 3^\infty) = 1\\
      0 & \text{otherwise,}
    \end{cases}    
  \end{align*}
  and $\rhoY(p,\m) = \rho(p,\m)$ for all $p \geqslant 5$ and all $\m$. 
\end{definition}

The functions $F_\Delta$ and $\rho$ are the counterparts of one another from an adelic perspective: 
$F_\Delta$ measures size at the infinite place, and
$\rho(p,\cdot)$
at the finite place $p$.

We are now ready to state our main result.

\begin{theorem}[Conductor distribution]
\label{thm:conductor_distribution_FN}
  Let $\cFQ(H)$, $\cFY(H)$, $F_\Delta$, $\rho$, and $\rhoY$ be as in \cref{cF_def,cF_tilde_def,FDelta_def,rho_def,rho_tilde_def}. For any $\lambda_1 > \lambda_0 > \tfrac{4464}{\log H}$,
  \begin{align*}
    &\frac{\#\!\left\{ E \in \cFQ(H) \,:\, \lambda_0 < \frac{N_E}{H} < \lambda_1 \right\}}{\#\cFQ(H)}\\
    &\hspace{1cm}= \frac{\zeta(10)}{\zeta(2)} \sum_{\m = 1}^\infty \big(F_\Delta(\m\lambda_1) - F_\Delta(\m\lambda_0) + F_\Delta(-\m\lambda_0) - F_\Delta(-\m\lambda_1)\big)\prod_p \frac{\rho(p,\m)}{1 - p^{-2}}\\
    &\hspace{1cm}\quad
    + O( (\log H)^{-1+\eps}),
    \shortintertext{and}
    &\frac{\#\!\left\{ E \in \cFY(H) \,:\, \lambda_0 < \frac{N_E}{H} < \lambda_1 \right\}}{\#\cFY(H)}\\
    &\hspace{1cm}= \frac{\zeta^{(6)}(10)}{\zeta^{(6)}(2)} \sum_{\m = 1}^\infty \big(F_\Delta(\m\lambda_1) - F_\Delta(\m\lambda_0) + F_\Delta(-\m\lambda_0) - F_\Delta(-\m\lambda_1)\big)\cdot\rhoY(2,\m)\rhoY(3,\m)\prod_{p \geqslant 5} \frac{\rhoY(p,\m)}{1 - p^{-2}}\\
    &\hspace{1cm}\quad+ O( (\log H)^{-1+\eps}).
  \end{align*}
\end{theorem}

In addition to being of intrinsic interest, \cref{thm:conductor_distribution_FN} is a tool for converting statistical results for height-ordered families of elliptic curves, which are plentiful, into statements about conductors, which are very elusive.
For example,
in \cite{ratiosconjecture}
the family $\cFY$ is shown to exhibit murmurations, a statistical phenomenon which revolves entirely around conductors, by leveraging the results of \cite{DHP}, whose proofs rely critically on the fact that $\cFY$ is height-ordered. \Cref{thm:conductor_distribution_FN} is the key tool that allows \cite{DHP}'s results to be useful in \cite{ratiosconjecture}. 

\Cref{thm:conductor_distribution_FN}'s right hand side is simple to compute: for any $\lambda_1 > \lambda_0 > 0$ the sum over $\m$ is finite, because the summand is $0$ whenever $\m\lambda_0 > 496$, and the product over $p$ is finite, because $\rho(p,\m) = 1 - p^{-2}$ whenever $p \nmid 6n$.

\Cref{fig:thm1} plots \cref{thm:conductor_distribution_FN}'s right hand sides and their derivatives, the aforementioned continuous distributions which fail to be differentiable infinitely often,
for both $\cFQ$ (red) and $\cFY$ (blue).
The functions plotted are essentially the cumulative distribution functions (left) and histograms (right) of the multisets
$\{N_E/H \,:\, E \in \cFQ(H)\}$ and $\{N_E/H \,:\, E \in \cFY(H)\}$ as $H \to \infty$. 

\begin{figure}[H]
\begin{minipage}[t]{0.5\linewidth}
\centering
\includegraphics[width=1.00\linewidth,valign=t,left]
                {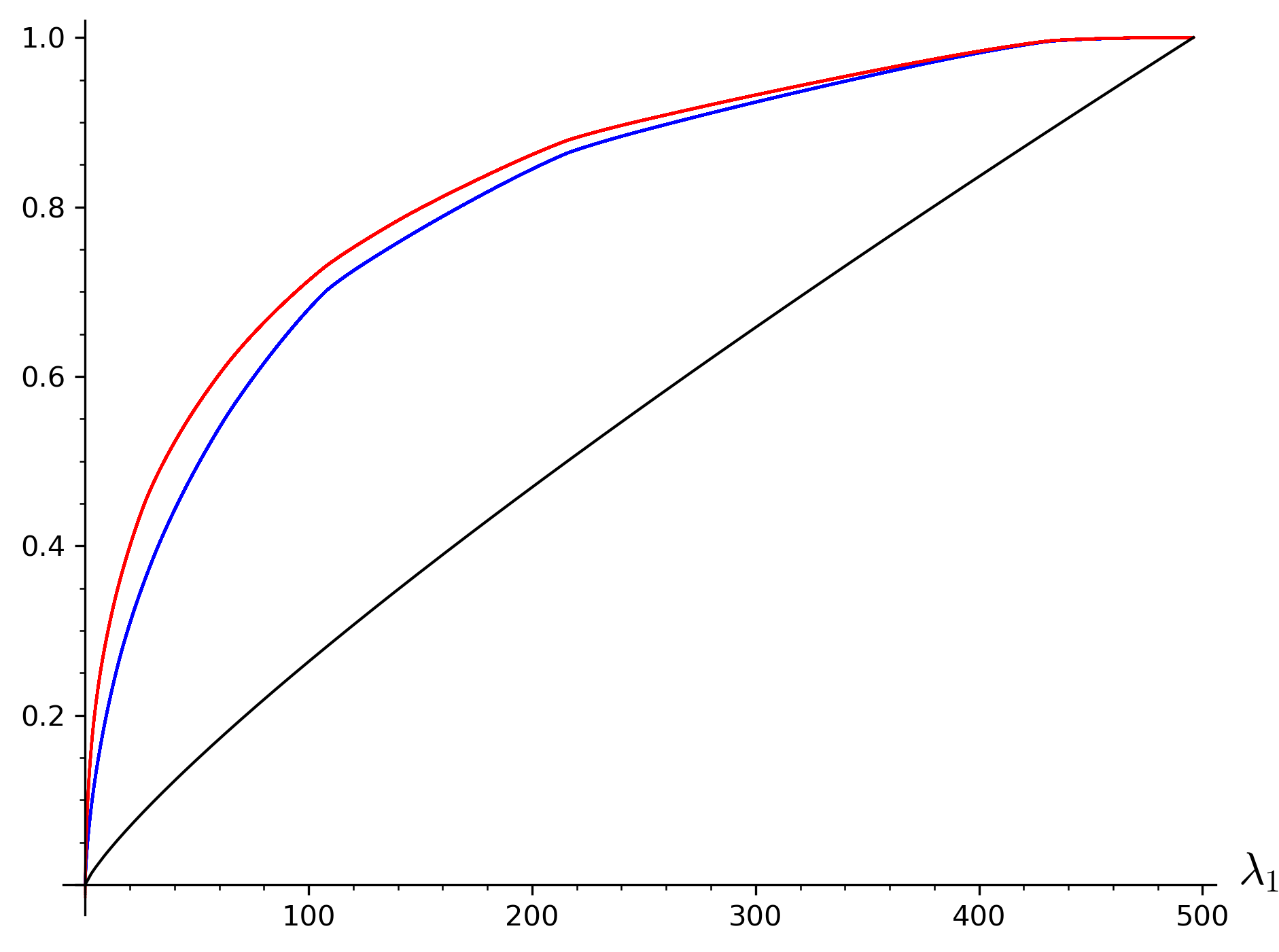}%
\end{minipage}%
\hfill
\begin{minipage}[t]{0.5\linewidth}
  \centering
  \includegraphics[width=1.00\linewidth,valign=t,left]
                  {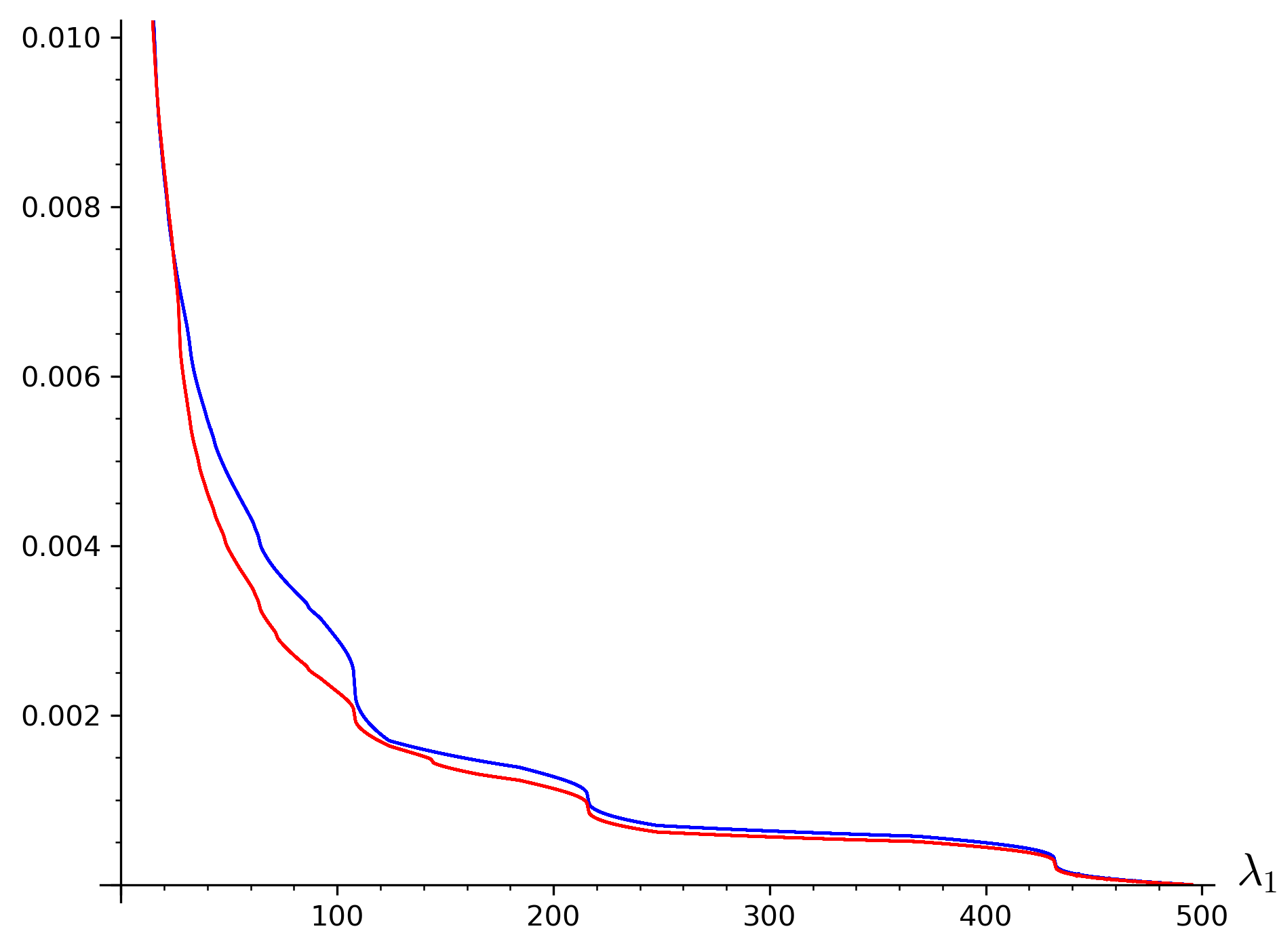}%
\end{minipage}
\begin{figurecap}
  \label{fig:thm1}~
  
  \hspace{0.5cm}Left: Main term on the right hand side of \cref{thm:conductor_distribution_FN} with $\lambda_0 = 0$
  (red: $\cFQ$, blue: $\cFY$), as well as the function
  $(\lambda_1 / 496)^{\frac{5}{6}}$ (black).
  
  \hspace{0.5cm}Right: Derivative with respect to $\lambda_1$ of the main term of \cref{thm:conductor_distribution_FN} (red: $\cFQ$, blue: $\cFY$).
  
  The identity \eqref{eq:fig_identity} was used generate these plots. The code is available at \cite{github_conductors}.
\end{figurecap}
\vspace{-\baselineskip}
\end{figure}

\Cref{thm:conductor_distribution_FN} also has applications to enumerating elliptic curves by conductor. As discussed above, this topic is very interesting and very difficult.
The classical question of counting by conductor has revolved around giving upper bound for the cardinality 
$\#\{E \in \cF(H) \,:\, N_E < \x\}$ as $H \to \infty$ with $\x$ fixed. Based on \cite[\S 4]{watkins} it is commonly believed that $\#\{E \in \cF(\infty) \,:\, N_E < \x\} \sim c\x^{\frac{5}{6}}$ for some $c > 0$ \cite[\S 1]{SSW}. This kind of asymptotic is known for certain large subfamilies of elliptic curves by work of Shankar--Shankar--Wang \cite{SSW}. The current best result for the family of all elliptic curves over $\Q$ is by Duke--Kowalski \cite[Prop.\ 1]{duke_kowalski}: the number of elliptic curves with conductor less than $\x$ is $\ll \x^{1+\eps}$.

\Cref{thm:tail_estimate} bounds the size of the set $\#\{E \in \cF(H) \,:\, N_E < \x\}$ for $H$ much larger than $\x$, but not infinite. I.e., it answers a quantitative version of the cancellation question:
Pick $\delta > 1$.
How many elliptic curves have $N < \x$ and $H < \x^\delta$?
We prove

\begin{theorem}[Counting by conductor]
  \label{thm:tail_estimate}
  \begin{align*}
    \x^{\frac{5}{6}} \ll \#\!\left\{E \in \cF(H) \,:\, N_E < \x\right\} \ll \x^{\frac{5}{6}} \left(\frac{H}{\x}\right)^{\frac{35}{54}}H^{\frac{7}{324}+\eps} + H^\half
    .
  \end{align*}
\end{theorem}

\Cref{thm:tail_estimate} is 
the strongest bound currently known for enumerating elliptic curves over $\Q$ with conductor at most $\x$ and height at most $H$ when
\begin{alignat*}{3}
  &H^{\frac{217}{264}+\eps} \approx H^{0.8220} &&\ll \x \ll \, H^{\frac{53}{60} - \eps} \,\approx H^{0.8833}
  \\
  \Longleftrightarrow\quad
  &\, \x^{\frac{60}{53} + \eps} \,\approx \x^{1.1321} &&\ll H \ll \x^{\frac{264}{217} - \eps} \approx \x^{1.2165}
  .
  \phantom{\quad\Longleftrightarrow}
\end{alignat*}
For $\delta > 264/217$, the aforementioned bound of Duke--Kowalski is better, and for $\delta < 60/53$, bounding by $\#\cF(H)$ is better.
\Cref{thm:tail_estimate} also holds with $\cFQ$ replaced by $\cFY$.

The final result we present establishes 
optimal bounds
for the number of elliptic curves $E \in \cF(H)$ with $N_E \leqslant \frac{4464 H}{\log H}$, the complement of the region 
admissible in \cref{thm:conductor_distribution_FN}.
Put otherwise,
in the range $N \leqslant \frac{4464 H}{\log H}$, \cref{thm:conductor_distribution_FN} cannot tell you the distribution of conductors, but \cref{thm:small_conductors} tells you there are exactly as many curves as you'd expect. 

\begin{theorem}
  \label{thm:small_conductors}
  For any $\lambda > \frac{4464}{\log H}$,
  \begin{align*}
    &\lambda^{\frac{5}{6}} \ll \frac{\#\!\left\{E \in \cF(H) \,:\, N_E < \lambda H\right\}}{\#\cF(H)} \ll \lambda^{\frac{5}{6}},
  \end{align*}
  and idem for $\cFY$.
\end{theorem}

Because \cref{thm:conductor_distribution_FN}'s hypothesis and error term both involve powers of $\log H$, the precise bound of \cref{thm:small_conductors}
will frequently be useful.
For instance, \cref{thm:small_conductors} is critical in establishing the main result of \cite{ratiosconjecture},
the application of our results that we discussed above; in contrast, any bound containing a factor of $H^\eps$ or similar would have had no hope of being helpful.

\section*{Acknowledgements}
We thank Jerry Wang very much for their insightful and meticulous feedback.


\section{Outline}\label{sec:outline}

Our plan for studying the distribution of $N/H$
will be to use the discriminants $\Delta = -16(4a^3 + 27b^2)$ of elliptic curves $E_{a,b} \in \cF(H)$\footnote{This section limits its discussion to the family $\cF$. The family $\cFY$ is handled in the same way.} as a stepping stone,
writing 
\begin{align}
  \label{eq:adelic}
  \frac{N}{H} \,=\, \frac{\Delta}{H} \cdot \frac{N}{\Delta} \,=\, \frac{\Delta}{H} \prod_p \frac{N_p}{\Delta_p}
  ,
\end{align}
where $N_p \coloneqq \gcd(N,p^\infty)$ and similarly for $\Delta_p$.

The decomposition \eqref{eq:adelic} emphasizes an adelic perspective on the ratio $N/H$. It is immediate from the relations $\Delta = -16(4a^3 + 27b^2)$ and $N\mid \Delta$ that $N/H \in (0,496]$,
and
the factor by which $N$ is smaller than $496H$ is a product of the ``archimedean'' factor $\Delta/H$ and the ``non-archimedean'' factors $N_p/\Delta_p$.

From this perspective, the well-known heuristics of Brumer--McGuinness \cite{bm}, relating height to discriminant, and Watkins \cite{watkins}, relating discriminant to conductor, can be viewed as proposing that the factors in the decomposition \eqref{eq:adelic} behave like appropriately chosen independent random variables on each completion of $\Q$.

To elaborate, consider the image of the box $B \coloneqq \{(a,b) \in \Z^2 \,:\, |a|^3,|b|^2 < H\}$ under the diagonal map $\delta: \Z \to \A_\Z$.
The Brumer--McGuinness heuristic 
approximates the image at the infinite place by
\begin{align*}
  \delta(B)_\infty \approx \{(x, y) \in \R^2 \,:\, |x|^3, |y|^2 < H\} \eqqcolon \Omega_\infty
\end{align*}
in the sense of weak convergence of the counting measure. I.e., with $dx\,dy$ denoting the Lebesgue measure on $\R^2$, the 
heuristic is that, for reasonable functions $f$,
\begin{align}
  \label{eq:bm}
  \sum_{(a,b) \in B} f(a,b) \approx \iint_{\Omega_\infty} f(x,y)\,dx\,dy
  .
\end{align}
The classical example is taking $f(x,y) = \one{16|4x^3 + 27y^2| < X}$ so as to estimate of the number of short Weierstrass equations with $|\Delta| < X$, obtaining 
\begin{align}
  \label{eq:bm_eval}
  \#\{(a,b) \in \Z^2 \,:\, 16|4a^3 + 27b^2| < X\}
  \approx
  \frac{1 + \sqrt{3}}{120} \frac{\Gamma\!\left(\half\right) \Gamma\!\left(\frac{1}{3}\right)}{\Gamma\!\left(\frac{5}{6}\right)} X^{\frac{5}{6}}
  .
\end{align}

Watkins's heuristic \cite[\S 4.1]{watkins} reasons similarly for the finite places, and additionally supposes an independence between the local factors of \eqref{eq:adelic}, basically positing that the diagonal map $\delta$ with the infinite place omitted is sufficiently close to being ergodic when acting on $\cS(D) \coloneqq \{(a,b) \in \Z \,:\, 4a^3 + 27b^2 \neq 0,\, |\Delta_{E_{a,b}}| < D\}$ as $D \to \infty$. Let $\mu = \prod_p \mu_p$ denote the Haar measure on $\hat{\Z}^2$.
One can formulate Watkins's heuristic in a way analogous to \eqref{eq:bm}:
\begin{align}
  \label{eq:watkins}
  \frac{1}{\#\cS(D)}\sum_{(a,b) \in \cS(D)} f(a,b) \approx \iint_{\hat{\Z}^2} f(x,y)\,d\mu
  .
\end{align}
Watkins's application takes $f(a,b) = \one{|\Delta_{E_{a,b}}| / N_{E_{a,b}} = \m}$ for each $\m \in \Z_{>0}$,
and then suggests
\begin{align}
  \label{eq:bmw}
  \#\{E/\Q \,:\, N_E < X\} \approx \sum_{\m=1}^\infty \#\{E/\Q \,:\, |\Delta_E| <\m X\} \prod_p \iint_{\Z_p^2} \onelr{\left|\frac{\Delta_{E_{x,y}}}{N_{E_{x,y}}}\right|_p = |\m|_p}d\mu_p
  ,
\end{align}
using the Brumer--McGuinness heuristic \eqref{eq:bm_eval} to estimate $\#\{E/\Q \,:\, |\Delta_E| <\m X\}$. (To be precise, we are presenting an analogue of the Brumer--McGuinness--Watkins pertaining to short Weierstrass models specifically. Some extra steps are needed to handle the family of globally minimal Weierstrass equations; see \cite{bm,watkins}.)

This paper's approach to proving \cref{thm:conductor_distribution_FN} is in essence a rigorous version of the heuristic reasoning leading to \eqref{eq:bmw}. Looking at \eqref{eq:watkins} for $f(a,b)$ of the form $\one{|\Delta_{E_{a,b}}| / N_{E_{a,b}} = \m}$ is tantamount to determining a distribution for the product \eqref{eq:adelic} over the finite places. The inclusion of the factor $\#\{E/\Q \,:\, |\Delta_E| <\m X\}$ in \eqref{eq:bmw} combined with \eqref{eq:bm} recovers the full product \eqref{eq:adelic}.

The assumption of independence, i.e.\ that one can factor the joint distribution over $(a,b) \in B$ of \eqref{eq:adelic}'s local ratios $(\Delta/H, N_2/\Delta_2, N_3/\Delta_3, \dots)$ as a product over the places of $\Q$, is a serious one.
In fact, the crux of our proof is that one specifically cannot factor the joint distribution this way! Indeed, for $|a|^3, |b|^2 < H$, the bound $-496H < \Delta < 64H$ implies that, for all $\lambda > 0$,
\begin{align}
  \label{eq:balance}
  \text{if } N > \lambda H, \text{ then } \frac{|\Delta|}{N} < \frac{496}{\lambda}
  ,
\end{align}
i.e.,
for $E_{a,b} \in \cF(H)$, if $N$ is a sufficiently large fraction of $H$, then $\Delta/N$ cannot be too big.

The significance of \eqref{eq:balance} is, if we go on to specify $\gcd(\Delta/N, C)$ for some $C$ which is divisible by all integers less than $496/\lambda$, then in fact we have specified the value of $|\Delta|/N$ exactly;
the valuations of $\Delta/N$ are intertwined.
Without this technique we would not have been able to say anything at all about the actual distribution of conductors, and would have merely obtained upper bounds instead.
This highlights the distinction between reporting a distribution like \cref{thm:conductor_distribution_FN}'s, vs.\ only bounds or statistics such as the zeroth moment.

The paper's key tool, which 
quantifies the extent to which $\delta(\cF(H))$ equidistributes in $\Omega_\infty \times \hat{\Z}$,
is the following. 

\begin{lemma}
  \label{disc_dist}
  Let $H \in \R_{>0}$, let $\lambda_0 < \lambda_1 \in \R$, let $\QM \in \Z_{>0}$, and let $S_\QM \subseteq (\Z/\QM\Z)^2$.
  Suppose that, for all $(a,b) \in S_\QM$, there does not exist an integer $d$ such that $\gcd(d,\QM) > 1$ and $\gcd(d^4,\QM) \mid \gcd(a,\QM)$ and $\gcd(d^6,\QM) \mid \gcd(b,\QM)$. Then
  \begin{align*}
    &\#\left\{ E_{a,b} \in \cF(H) \,:\, (a\mod \QM,\, b\mod \QM) \in S_\QM, \,\,\lambda_0 < \frac{\Delta_E}{H} < \lambda_1 \right\} \\
    &\hspace{5cm} = \frac{4H^{\frac{5}{6}}}{\zeta^{(\QM)}(10)} \frac{\#S_\QM}{\QM^2} \left(F_\Delta(\lambda_1) - F_\Delta(\lambda_0)\right) + O\!\left(H^\half \frac{\#S_\QM}{\QM} + \#S_\QM\right).
  \end{align*}
\end{lemma}
~\\

\Cref{disc_dist} can be viewed as a precise form of the Brumer--McGuinness heuristic which additionally
\begin{itemize}
\item
  pertains specifically to elliptic curves $E_{a,b} \in \cF(H)$, as opposed to all $(a,b) \in \Z$, and
\item
  allows one to incorporate congruence conditions on $a$ and $b$.
\end{itemize}

Congruence conditions allow us to prescribe values for the local factors in \eqref{eq:adelic} at the finite places. For each prime $p$ and local reduction type $\T$, the condition that $E \mod p$ has type $\T$ is given by a congruence modulo some power of $p$, in the sense that there exist $k \in \Z_{\geqslant 0}$ and $S_\T \subseteq (\Z/p^k\Z)^2$ such that $E_{a,b} \mod p$ has type $\T$ if and only if $(a \mod p^k, b \mod p^k) \in S_\T$.
\Cref{sec:local} determines these congruence conditions. For example,
\begin{align}
  \label{eq:example_congs}
  \begin{aligned}
  5 &\;\!\nmid\:\! \Delta/N
  \text{ iff }
  5^2 \nmid 4a^3 + 27b^2 \text{ or } 5 \parallel b,
  \\
  5^8 &\parallel \Delta/N
  \text{ iff }
  5^4 \mid a \text{ and } 5^5 \parallel b,
  \\
  2^2 &\parallel \Delta/N
  \text{ iff }
  b = 1\mod 4 \text{ and } a = \text{$0$ or $3$}\mod 4,
  \\
  3^{1009} &\parallel \Delta/N
  \text{ iff }
  3^3 \parallel a,b \text{ and } 3^{1010} \parallel 4a^3 + 27b^2, \text{ or } 3 \parallel a \text{ and } 3 \nmid b \text{ and } 3^{1011} \parallel 4a^3 + 27b^2.
  \end{aligned}
\end{align}
The congruence conditions on $a$ and $b$ in the examples above can be taken to be restrictions modulo $5^2$, $5^6$, $2^2$, and $3^{1009}$ respectively, but not modulo any smaller powers. The proportions $p^{-2k}\#S_\T$ of congruence classes which satisfy these conditions are $1 - 5^{-2}$, $5^{-9}(1 - 5^{-1})(2 - 2 \cdot 5^{-1})$, $2^{-3}$, and $40 \cdot 3^{-1011}$ respectively. The ratio $\Delta_p / N_p$ depends only on the reduction type of $E \mod p$ (at least for minimal $E$; at $2$ and $3$ we have some non-minimal cases to handle too, with the idea being the same).

With complete data of the sort given in the \eqref{eq:example_congs}, we will ultimately be able to implement Brumer--McGuinness--Watkins's idea leading to \eqref{eq:bmw}.
There are technical hurdles to overcome. We readily observe that prescribing $\gcd(\Delta / N, C) = n$ for large $C$ requires the imposition of a congruence condition modulo a correspondingly large modulus $\QM$.
Quantitatively,
for $\q \geqslant 496/\lambda$ set
\begin{align*}
  \QD \coloneqq \prod_{\phantom{{}_*}p \,<\, \q} p^{\left\lfloor \frac{\log \q}{\log p} \right\rfloor + 2}
  \quad\quad\text{and}\quad\quad
  C \coloneqq \prod_{\phantom{{}_*}p \,<\, \q} p^{\left\lfloor \frac{\log \q}{\log p} \right\rfloor}
  .
\end{align*}
Examining \cref{table:redp,table:red2_all,table:red3_all} shows that specifying $\gcd(\Delta/N, C) = \m$ for any $\m \mid C$ is possible using a congruence modulo $\QD$.

The error term in \cref{disc_dist} grows with the modulus $\QM$ of the congruence condition imposed on $a$ and $b$. This creates a tension: we want $\QM$ small so that the main term of \cref{disc_dist} isn't dwarfed by the error term, yet we want $\QM$ large so that we may effectively estimate $N$ using the idea \eqref{eq:balance} begets. Navigating this tension is the paper's main struggle.

Wtih some technical analysis, we'll determine the distribution of conductors of $\cF(H)$ in what will turn out to be the range $N > 4464H/\log H$.
What remains, then, is to estimate how many conductors of $\cF(H)$ fall outside this range.

We proceed by splitting the complementary range $N < 4464H/\log H$ into two parts, according to whether or not $N > H^{\frac{5}{6}+\eps}$. We handle the larger conductors in \cref{FN_bigm} and the smaller ones in \cref{N_56_bound}. \Cref{FN_bigm} may be interesting in its own right, and its proof features a trick which, in the context of proving \cref{thm:conductor_distribution_FN}, amounts to swapping one set of local conditions used in \cref{disc_dist} for another at an opportune moment.

At this point \cref{thm:conductor_distribution_FN} will have been established.
With it then available as a tool, the bulk of the work will be behind us.

\Cref{thm:small_conductors} will be proved next. The proof leverages \cref{thm:conductor_distribution_FN} heavily. Essentially, the idea is that \cref{thm:conductor_distribution_FN} describes the complementary range completely, and it will be enough to simply look at what mass that description gives to its complementary range, i.e.\ the range of \cref{thm:small_conductors}. This mass is calculated in \cref{FN_small_bound} to be the expected amount, and is indeed larger than \cref{thm:conductor_distribution_FN}'s error term. We find it interesting that it is not at all clear from \cref{fig:thm1} that the expected upper bound $\lambda^{\frac{5}{6}}$ holds.

The proof of \cref{thm:tail_estimate} will involve adaptations and variations of ideas which will have at that point been encountered elsewhere, as well as some straightforward optimization to balance error terms.


\section{Discriminant distributions with local restrictions}\label{sec:disc}

The main result of this section is \cref{disc_dist}, 
whose significance is discussed in \cref{sec:outline}.
We begin with a version pertaining to all integral short Weierstrass equations, as opposed to only those in $\cF(H)$.

\begin{lemma}\label{disc_to_int}
  Let $H \in \R_{>0}$, let $\lambda_0 < \lambda_1 \in \R$, let $\QM \in \Z_{>0}$, and let $S_\QM \subseteq (\Z/\QM\Z)^2$.
  \begin{align*}
    &\#\left\{|a| < H^\third, |b| < H^\half \,:\, (a\mod \QM, b\mod \QM) \in S_\QM, \,\,\lambda_0 < -16\frac{4a^3 + 27b^2}{H} < \lambda_1 \right\} \\
    &\hspace{7cm}= 4H^{\frac{5}{6}}\frac{\#S_\QM}{\QM^2}\left(F_\Delta(\lambda_1) - F_\Delta(\lambda_0)\right) + O\!\left(H^\half\frac{\#S_\QM}{\QM} + \#S_\QM\right).
  \end{align*}
\end{lemma}
\begin{proof}
  Tile $\left\{ |a| < H^\third, |b| < H^\half \right\}$ with $\QM \times \QM$ squares. In each square, $\#\{(a\mod \QM, b\mod \QM) \in S_\QM\} = \#S_\QM$. The curves $$-16\frac{4a^3 + 27b^2}{H} \in \{\lambda_0, \lambda_1\}$$ intersect $\ll \QM^{-1}H^\half + \QM^{-1}H^\third + 1$ squares. This is also the number of incomplete squares around the perimeter. Each of these squares contributes a discrepancy between the main terms on the left and right hand sides of the lemma of size at most $2\#S_\QM$.
\end{proof}

Incorporating the restriction $p^4 \mid a \Rightarrow p^6 \nmid b$ from the definition of $\cF(H)$ into \cref{disc_to_int} leads to \cref{disc_dist}, which we now prove.

\begin{proof}[Proof of \cref{disc_dist}]
  \cite[Proof of lemma 3.2]{young:elliptic_curves} gives the following variant of M\"{o}bius inversion:
  \begin{align*}
    \sum_{\substack{d \in \Z_{>0} \\ d^4 \mid a \text{ and } d^6 \mid b}} \mu(d) = \begin{cases}1 & \text{if there does not exist a prime $p$ such that $p^4 \mid a$ and $p^6 \mid b$} \\ 0 & \text{otherwise.}\end{cases}
  \end{align*}

  This identity implies that, for any $S \subseteq \Z^2$ and any function $f: S \to \CC$,
  \begin{align}
    \sum_{\substack{(a,b) \in S \\ p^4 \mid a \,\Rightarrow\, p^6 \nmid b}} f(a,b)
    &= \sum_{(a,b) \in S} f(a,b) \sum_{\substack{d \in \Z_{>0} \\ d^4 \mid a \text{ and } d^6 \mid b}} \mu(d) \nonumber\\
    &= \sum_{d \in \Z_{>0}} \mu(d) \sum_{\substack{(\alpha, \beta) \in \Z^2 \\ (\alpha d^4, \beta d^6) \in S}} f(\alpha d^4, \beta d^6). \label{young_moebius}
  \end{align}

  Let
  $$S_\QM^{(d)} \coloneqq \{ (a,b) \in (\Z/\QM\Z)^2 \,:\, (ad^4, bd^6) \in S_\QM \}.$$
  Note that
  \begin{align}
    \text{$\#S_\QM^{(d)} = \#S_\QM$ when $\gcd(d,\QM) = 1$, and $\#S_\QM^{(d)} = 0$ when $\gcd(d,\QM) > 1$.} \label{SQd_observation}
  \end{align}
  The second observation of \eqref{SQd_observation} follows from \cref{disc_dist}'s assumptions on $S_\QM$.
  
  Taking
  $$S = \{ (a,b) \in \Z^2 \,:\, E_{a,b} \in \cF(H),\,\, (a\mod \QM,\, b\mod \QM) \in S_\QM\}$$
  in \eqref{young_moebius} yields
  \begin{align}
    \sum_{\substack{|a| \,<\, H^\third \\ |b| \,<\, H^\half \\ (a\mod \QM,\, b\mod \QM) \,\in\, S_\QM \\ p^4\mid a \,\Rightarrow\, p^6\nmid b}} f(a,b) = \sum_{\substack{0 \,<\, d \,<\, H^{\frac{1}{12}} \\ \gcd(d,\QM) \,=\, 1}} \mu(d) \sum_{\substack{|a| \,<\, (d^{-12}H)^\third \\ |b| \,<\, (d^{-12}H)^\half \\ (a\mod \QM,\, b\mod \QM) \,\in\, S_\QM^{(d)} }} f(ad^4, bd^6), \label{moebius_cF}
  \end{align}
  where the condition $\gcd(d,\QM) = 1$ could be inserted freely as a result of \eqref{SQd_observation}.

  \Cref{disc_dist} follows from applying \eqref{moebius_cF} to the function
  \begin{align*}
    f(a,b) = \begin{cases}1 & \text{if } \lambda_0 < -16\frac{4a^3 + 27b^2}{H} < \lambda_1 \\ 0 & \text{otherwise,}\end{cases}
  \end{align*}
  combining with \cref{disc_to_int}, and using \eqref{SQd_observation}.
\end{proof}

\begin{remark}
  For arbitrary $S \subseteq (\Z/\QM\Z)^2$ and $d \in \Z_{>0}$ an elementary calculation shows that
  \begin{align*}
    &\#\{(a,b) \in (\Z/\QM\Z)^2 \,:\, (a d^4, b d^6) \in S\}\\
    &\hspace{2cm}= \gcd(d^4,\QM) \gcd(d^6,\QM) \cdot \#\{ (\alpha, \beta) \in S \,:\, \text{$\gcd(d^4,\QM) \mid \gcd(\alpha,\QM)$ and $\gcd(d^6,\QM) \mid \gcd(\beta,\QM)$}\}.
  \end{align*}
\end{remark}

\begin{remark}
  Let $r,t$ be as in \cref{cF_tilde_def}. If $6\mid \QM$ and $a = r \ \mathrm{mod}\ 6$ and $b = t \ \mathrm{mod}\ 6$ for all $(a,b) \in S_\QM$, then \cref{disc_dist} holds with $\cF$ replaced by $\cFY$.
\end{remark}

The proof of \cref{disc_dist} with $f(a,b) = 1$, combined with \cref{disc_to_int} for $\lambda_0 = -\infty$, $\lambda_1 = \infty$, and either $\QM = 1$, or $\QM = 6$ and $S_\QM = \{(a,b) \,:\, a = r \ \mathrm{mod}\ 6,\,\, b = t \ \mathrm{mod}\ 6\}$, yields
\begin{lemma}\label{FH_size}
  \begin{align*}
    \#\cF(H) = \frac{4H^{\frac{5}{6}}}{\zeta(10)} + O(H^\half) \quad\quad\text{and}\quad\quad \#\cFY(H) = \frac{H^{\frac{5}{6}}}{9\zeta^{(6)}(10)} + O(H^\half).
  \end{align*}
\end{lemma}


\section{Reduction types}\label{sec:local}

\Cref{disc_dist} takes as input the modulus $\QM$ of a congruence condition, and a \textit{proportion} $\# S_\QM / \QM^2$ reflecting how often the associated congruence condition is satisfied. The lemma's error term is such that it's desirable for $\QM$ to be as small as possible. In this section we determine the minimal moduli and the proportions for the congruence conditions associated to reduction types at each prime $p$.

We handle the cases $p = 2$, $p = 3$, and $p \geqslant 5$ separately.
\Cref{table:redp} lists the reduction information at $p \geqslant 5$ for both $\cFQ$ and $\cFY$. 
\Cref{table:red2_all,table:red3_all} pertain specifically to the family $\cF$, at the primes $p = 2$ and $p = 3$ respectively. \Cref{fig:red23} illustrates the information contained in these two tables. \Cref{local_densities_3,local_densities_2} describe reduction at $p = 3$ and $p = 2$ for the family $\cFY$.

The \cref{table:redp,table:red2_all,table:red3_all} list all possible reduction types among $E_{a,b}$ in $\cF$. 
For each reduction type $\T$, the tables also list
\begin{itemize}
\item exhaustive criteria on the valuations $v_p(a)$, $v_p(b)$, and $v_p(\Delta)$ for $E_{a,b}$ with reduction type $\T$ (where $v_p(x)$ is defined by $p^{v_p(x)} \!\parallel\! x$),
\item the largest power $v_p(N)$ of $p$ dividing the conductor of $E$ with reduction type $\T$,
\item the least integer $\QT$ such that there exists $S_\T \subseteq (\Z/\QT\Z)^2$ with the property that $(a \mod \QT, b \mod \QT) \in S_T$ if and only if $E_{a,b}$ has reduction type $\T$ at $p$, and
\item the proportion $\QT^{-2} \#S_\T$ of residue classes for which $E_{a,b}$ has reduction type $\T$ at $p$.
\end{itemize}



The condition $p^4\nmid a$ or $p^6\nmid b$ in the definition of $\cFQ$ ensures that $E_{a,b} \in \cFQ$ is minimal at every prime $p \geqslant 5$. It is not true that $E_{a,b}$ is necessarily minimal at $p = 2$ or $3$. It follows from \cite[Tableaux II et IV]{papadopoulos} that if $E_{a,b}$ is not minimal at $p = 2$ or $3$, then $E_{a,b}$ has either good or multiplicative reduction, and the $p$-part of the minimal discriminant of $E_{a,b}$ is equal to the $p$-part of $p^{-12}\Delta_E$. 

In \cref{table:red2_all,table:red3_all} we write e.g.\
``$3^{12}\mathrm{I}_0$'' to denote type $\mathrm{I}_0$ reduction with the additional condition that the discriminant of $E_{a,b}$ is $3^{12}$ times its minimal discriminant. All cases which do not feature ``$2^{12}$'' or ``$3^{12}$'' in the notation are minimal.
We also write e.g.\ 
``$\mathrm{II}(5)$'' to denote type $\mathrm{II}$ reduction with the additional condition that $v_p(N) = 5$.


The case $p^4 \!\mid\! a \;\text{and}\; p^6 \!\mid\! b$ is also included in \cref{table:redp,table:red2_all,table:red3_all}, even though the short Weierstrass models $E_{a,b} \in \cF(H)$ 
never fall into this case by definition.

\begin{table}[H]
\begin{tabular}{|l|r|r|r||c|c|l|}
  \multicolumn{1}{c}{} & \multicolumn{3}{c}{Conditions} & \multicolumn{3}{c}{} \\
  \hline
  Type & $v_p(a)$ & $v_p(b)$ & $v_p(\Delta)$ & $v_p(N)$ & $\QT$ & Proportion \rule{0pt}{1em}\\
  [0.2ex]
  \hline
  $\mathrm{I}_0$            & $\geqslant 0$ & $\geqslant 0$  & $0$ & $0$ & $p$ & $1 - p^{-1}$ \rule{0pt}{1em} \\ [0.2ex]
  $\mathrm{I}_\m,\; \m\geqslant 1$ & $0$ & $0$ & $\m$ & $1$   & $p^{\m+1}$ & $p^{-\m}(1 - p^{-1})^2$   \\ [0.2ex]
  $\mathrm{II}$             & $\geqslant 1$ & $1$ & $2$ & $2$   & $p^2$    & $p^{-2}(1 - p^{-1})$     \\ [0.2ex]
  $\mathrm{III}$            & $1$ & $\geqslant 2$ & $3$ & $2$   & $p^2$    & $p^{-3}(1 - p^{-1})$     \\ [0.2ex]
  $\mathrm{IV}$             & $\geqslant 2$ & $2$ & $4$ & $2$   & $p^3$    & $p^{-4}(1 - p^{-1})$     \\ [0.2ex]
  $\mathrm{I}_0^*$          &  &  & $6$ & $2$   & $p^4$    & $p^{-5}(1 - p^{-1})$      \\ [0.2ex]
  $\mathrm{I}_\m^*,\; \m\geqslant 1$ & $2$ & $3$ & $\m+6$ & $2$ & $p^{\m+4}$ & $p^{-\m-5}(1 - p^{-1})^2$  \\ [0.2ex]
  $\mathrm{IV}^*$            & $\geqslant 3$ & $4$ & $8$ & $2$   & $p^5$    & $p^{-7}(1 - p^{-1})$     \\ [0.2ex]
  $\mathrm{III}^*$           & $3$ & $\geqslant 5$ & $9$ & $2$   & $p^5$    & $p^{-8}(1 - p^{-1})$     \\ [0.2ex]
  $\mathrm{II}^*$            & $\geqslant 4$ & $5$ & $10$ & $2$  & $p^6$    & $p^{-9}(1 - p^{-1})$     \\ [0.2ex]
  $p^4 \!\mid\! a \;\text{and}\; p^6 \!\mid\! b$ & $\geqslant 4$ & $\geqslant 6$ & $\geqslant 12$ & --- & $p^6$ & $p^{-10}$  \\ [0.2ex]
  \hline
\end{tabular}
\begin{tablecap}\label{table:redp}
  Reduction information at $p \geqslant 5$ for the families $\cF$ and $\cFY$.
\end{tablecap}
\end{table}


\begin{table}[H]
\begin{tabular}{|l|r|r|r|c||c|c|l|}
  \multicolumn{1}{c}{} & \multicolumn{4}{c}{Conditions} & \multicolumn{3}{c}{} \\
  \hline
  Type & $v_2(a)$ & $v_2(b)$ & $v_2(\Delta)$ & Extra & $v_2(N)$ & $\QT$ & Proportion \rule{0pt}{1em}\\
  [0.2ex]
  \hline
  $\mathrm{II}(4)$ & $\geqslant 0$ & $0$ & $4$ & $\neg$\eqref{eq:p2_extra_condition_1}$\phantom{\neg}$ & $4$ & $2^2$ & $2^{-2}$ \rule{0pt}{1em} \\ [0.2ex]
  $\mathrm{II}(6)$ & $0$ & $\geqslant 1$ & $6$ & $\neg$\eqref{eq:p2_extra_condition_1}$\phantom{\neg}$ & $6$ & $2^2$ & $2^{-4}$ \\ [0.2ex]
                   & $\geqslant 1$ & $1$ & $6$ && $6$ & $2^2$ & $2^{-3}$ \\ [0.2ex]
  $\mathrm{II}(7)$ & $0$ & $1$ & $7$ && $7$ & $2^2$ & $2^{-4}$ \\ [0.2ex]
  $\mathrm{III}(3)$ & $0$ & $0$ & $4$ & \eqref{eq:p2_extra_condition_1} and $\neg$\eqref{eq:p2_extra_condition_2} & $3$ & $2^2$ & $2^{-4}$ \\ [0.2ex]
                    & $1$ & $0$ & $4$ & \eqref{eq:p2_extra_condition_1} & $3$ & $2^2$ & $2^{-4}$ \\ [0.2ex]
  $\mathrm{III}(5)$ & $0$ & $\geqslant 1$ & $6$ & \eqref{eq:p2_extra_condition_1} & $5$ & $2^2$ & $2^{-4}$ \\ [0.2ex]
  $\mathrm{III}(7)$ & $1$ & $2$ & $8$ && $7$ & $2^3$ & $2^{-5}$ \\ [0.2ex]
  $\mathrm{III}(8)$ & $1$ & $\geqslant 3$ & $9$ && $8$ & $2^3$ & $2^{-5}$ \\ [0.2ex]
  $\mathrm{IV}$ & $0$ & $0$ & $4$ & \eqref{eq:p2_extra_condition_1} and \eqref{eq:p2_extra_condition_2}& $2$ & $2^2$ & $2^{-4}$ \\ [0.2ex]
                & $\geqslant 2$ & $0$ & $4$ & \eqref{eq:p2_extra_condition_1} & $2$ & $2^2$ & $2^{-4}$ \\ [0.2ex]
  $\mathrm{I}_0^*(4)$ & $0$ & $1$ & $8$ & $\neg$\eqref{eq:p2_extra_condition_3}$\phantom{\neg}$ & $4$ & $2^4$ & $2^{-6}$ \\ [0.2ex]
                      & $\geqslant 2$ & $2$ & $8$ & $\neg$\eqref{eq:p2_extra_condition_3}$\phantom{\neg}$ & $4$ & $2^4$ & $2^{-6}$ \\ [0.2ex]
  $\mathrm{I}_0^*(5)$ & $0$ & $1$ & $9$ && $5$ & $2^4$ & $2^{-6}$ \\ [0.2ex]
  $\mathrm{I}_0^*(6)$ & $\geqslant 2$ & $3$ & $10$ && $6$ & $2^4$ & $2^{-6}$ \\ [0.2ex]
  $\mathrm{I}_1^*$ & $0$ & $1$ & $8$ & \eqref{eq:p2_extra_condition_3} and $\neg$\eqref{eq:p2_extra_condition_4} & $3$ & $2^4$ & $2^{-7}$ \\ [0.2ex]
                  & $2$ & $2$ & $8$ & \eqref{eq:p2_extra_condition_3} & $3$ & $2^4$ & $2^{-7}$ \\ [0.2ex]
  $\mathrm{I}_2^*(4)$ & $0$ & $1$ & $10$ & $\neg$\eqref{eq:p2_extra_condition_4}$\phantom{\neg}$ & $4$ & $2^5$ & $2^{-8}$ \\ [0.2ex]
  $\mathrm{I}_2^*(6)$ & $2$ & $\geqslant 5$ & $12$ & $\neg$\eqref{eq:p2_extra_condition_5}$\phantom{\neg}$ & $6$ & $2^5$ & $2^{-9}$ \\ [0.2ex]
  $\mathrm{I}_2^*(7)$ & $2$ & $4$ & $13$ && $7$ & $2^5$ & $2^{-9}$ \\ [0.2ex]
  $\mathrm{I}_3^*(4)$ & $0$ & $1$ & $11$ & $\neg$\eqref{eq:p2_extra_condition_4}$\phantom{\neg}$ & $4$ & $2^6$ & $2^{-9}$ \\ [0.2ex]
  $\mathrm{I}_3^*(5)$ & $2$ & $\geqslant 5$ & $12$ & \eqref{eq:p2_extra_condition_5} & $5$ & $2^5$ & $2^{-9}$ \\ [0.2ex]
  $\mathrm{I}_4^*(4)$ & $0$ & $1$ & $12$ & $\neg$\eqref{eq:p2_extra_condition_4}$\phantom{\neg}$ & $4$ & $2^{7}$ & $2^{-10}$ \\ [0.2ex]
  $\mathrm{I}_4^*(6)$ & $2$ & $4$ & $14$ && $6$ & $2^{5}$ & $2^{-10}$ \\ [0.2ex]
  $\mathrm{I}_\m^*(4),\; \m \geqslant 5$ & $0$ & $1$ & $\m+8$ && $4$ & $2^{\m+3}$ & $2^{-\m-6}$ \\ [0.2ex]
  $\mathrm{I}_\m^*(6),\; \m \geqslant 5$ & $2$ & $4$ & $\m+10$ && $6$ & $2^{\m+2}$ & $2^{-\m-6}$ \\ [0.2ex]
  $\mathrm{IV}^*$ & $0$ & $1$ & $8$ & \eqref{eq:p2_extra_condition_3} and \eqref{eq:p2_extra_condition_4} & $2$ & $2^4$ & $2^{-7}$ \\ [0.2ex]
                  & $\geqslant 3$ & $2$ & $8$ & \eqref{eq:p2_extra_condition_3} & $2$ & $2^4$ & $2^{-7}$ \\ [0.2ex]
  $\mathrm{III}^*(3)$ & $0$ & $1$ & $10$ & \eqref{eq:p2_extra_condition_4} & $3$ & $2^5$ & $2^{-8}$ \\ [0.2ex]
  $\mathrm{III}^*(5)$ & $3$ & $4$ & $12$ && $5$ & $2^5$ & $2^{-9}$ \\ [0.2ex]
  $\mathrm{III}^*(7)$ & $3$ & $5$ & $14$ && $7$ & $2^6$ & $2^{-10}$ \\ [0.2ex]
  $\mathrm{III}^*(8)$ & $3$ & $\geqslant 6$ & $15$ && $8$ & $2^6$ & $2^{-10}$ \\ [0.2ex]
  $\mathrm{II}^*(3)$ & $0$ & $1$ & $11$ & \eqref{eq:p2_extra_condition_4} & $3$ & $2^6$ & $2^{-9}$ \\ [0.2ex]
  $\mathrm{II}^*(4)$ & $\geqslant 4$ & $4$ & $12$ & $\neg$\eqref{eq:p2_extra_condition_6}$\phantom{\neg}$ & $4$ & $2^6$ & $2^{-10}$ \\ [0.2ex]
  $\mathrm{II}^*(6)$ & $\geqslant 4$ & $5$ & $14$ && $6$ & $2^6$ & $2^{-10}$ \\ [0.2ex]
  $2^{12}\mathrm{I}_0$ & $0$ & $1$ & $12$ & \eqref{eq:p2_extra_condition_4} & $0$ & $2^7$ & $2^{-10}$ \\ [0.2ex]
                      & $\geqslant 4$ & $4$ & $12$ & \eqref{eq:p2_extra_condition_6} & $0$ & $2^7$ & $2^{-10}$ \\ [0.2ex]
  $2^{12}\mathrm{I}_\m,\; \m \geqslant 1$ & $0$ & $1$ & $\m+12$ & 
  & $1$ & $2^{\m+7}$ & $2^{-\m-10}$ \\ [0.2ex]
  $2^4 \!\mid\! a \;\text{and}\; 2^6 \!\mid\! b$ & $\geqslant 4$ & $\geqslant 6$ & $\geqslant 16$ && --- & $2^6$ & $2^{-10}$ \\ [0.2ex]
  \hline
\end{tabular}
\begin{tablecap}\label{table:red2_all}
  Reduction information at $p = 2$ for the family $\cF$. 
\end{tablecap}
\vspace{-\baselineskip}
\end{table}

\begin{align}
  \label{eq:p2_extra_condition_1}
  &2^2 \mid b - b^2 + a^2 + a^3\\
  \label{eq:p2_extra_condition_2}
  &2^3 \mid \psi_3(a)\\
  \label{eq:p2_extra_condition_3}
  &\text{There exist $r,t \in \Z$ such that $2^5 \mid \psi_3(r)$ and $2^6 \mid \psi_2(r) - 4t^2$ simultaneously.}\\
  \label{eq:p2_extra_condition_4}
  &\text{There exists $r \in \Z$ such that $r = 1$ or $2 \ \mathrm{mod}\ 4$ and $2^5 \mid \psi_3(r)$ simultaneously.}\\
  \label{eq:p2_extra_condition_5}
  &2^4 \mid a - 4\\
  \label{eq:p2_extra_condition_6}
  &\text{There exist $r,t \in \Z$ such that $2^8 \mid \psi_3(r)$ and $2^8 \mid \psi_2(r) - 4t^2$ simultaneously.}
\end{align}

In \cref{table:red2_all,table:red3_all}, we additionally list extra conditions on $a$ and $b$ whenever $v_p(a)$, $v_p(b)$, and $v_p(\Delta)$ alone are insufficient to uniquely specify the reduction type. These are given in \cite[\S II]{papadopoulos} (except for \eqref{eq:p2_extra_condition_5}; see \cref{rem:papadopoulos_correction}).
In conditions \eqref{eq:p2_extra_condition_1}---\eqref{eq:p2_extra_condition_6}, $\psi_2(x) \coloneqq 4(x^3 + ax + b)$ and $\psi_3(x) \coloneqq 3x^4 + 6ax^2 + 12bx - a^2$ denote the $2$- and $3$-division polynomials of $E_{a,b}$.

\begin{align}
  \label{eq:p3_type_III_extra_condition}
  &3^2 \mid a + b^2 - 1\\
  \label{eq:p3_type_IIIstar_extra_condition}
  &3^2 \mid 3^{-2}a + 3^{-6}b^2 - 1
\end{align}




\begin{table}[H]
\begin{tabular}{|l|r|r|r|c||c|c|l|}
  \multicolumn{1}{c}{} & \multicolumn{4}{c}{Conditions} & \multicolumn{3}{c}{} \\
  \hline
  Type & $v_3(a)$ & $v_3(b)$ & $v_3(\Delta)$ & Extra & $v_3(N)$ & $\QT$ & Proportion \rule{0pt}{1em}\\
  [0.2ex]
  \hline
  $\mathrm{I}_0$   & $\geqslant 0$ & $\geqslant 0$ & $0$ && $0$ & $3$ & $2 \cdot 3^{-1}$ \rule{0pt}{1em} \\ [0.2ex]
  $\mathrm{II}(3)$ & $\geqslant 1$ & $0$ & $3$ & $\neg$\eqref{eq:p3_type_III_extra_condition}$\phantom{\neg}$ & $3$ & $3^2$ & $8 \cdot 3^{-4}$ \\ [0.2ex]
                   & $1$      & $1$ &     &                                        &     & $3^2$ & $4 \cdot 3^{-4}$ \\ [0.2ex]
  $\mathrm{II}(4)$ & $1$ & $0$ & $4$ && $4$ & $3^2$ & $4 \cdot 3^{-4}$ \\ [0.2ex]
  $\mathrm{II}(5)$ & $\geqslant 2$ & $1$ & $5$ && $5$ & $3^2$ & $2 \cdot 3^{-4}$ \\ [0.2ex]
  $\mathrm{III}$   & $\geqslant 1$ & $0$      & $3$ & \eqref{eq:p3_type_III_extra_condition} & $2$ & $3^2$ & $4 \cdot 3^{-4}$ \\ [0.2ex]
                   & $1$      & $\geqslant 2$ &     &                                        &     & $3^2$ & $2 \cdot 3^{-4}$ \\ [0.2ex]
  $\mathrm{IV}(3)$ & $1$ & $0$ & $5$ && $3$ & $3^3$ & $4 \cdot 3^{-5}$ \\ [0.2ex]
  $\mathrm{IV}(4)$ & $2$ & $2$ & $6$ && $4$ & $3^3$ & $4 \cdot 3^{-6}$ \\ [0.2ex]
  $\mathrm{IV}(5)$ & $\geqslant 3$ & $2$ & $7$ && $5$ & $3^3$ & $2 \cdot 3^{-6}$ \\ [0.2ex]
  $\mathrm{I}_0^*$ & $1$ & $0$      & $6$ && $2$ & $3^4$ & $4 \cdot 3^{-6}$ \\ [0.2ex]
                   & $2$ & $\geqslant 3$ & $6$ && $2$ & $3^3$ & $2 \cdot 3^{-6}$ \\ [0.2ex]
  $\mathrm{I}_\m^*,\; \m \geqslant 1$ & $1$ & $0$ & $\m+6$ && $2$ & $3^{\m+4}$ & $4 \cdot 3^{-\m-6}$ \\ [0.2ex]
  $\mathrm{IV}^*(3)$ & $\geqslant 3$ & $3$ & $9$ & $\neg$\eqref{eq:p3_type_IIIstar_extra_condition}$\phantom{\neg}$ & $3$ & $3^5$ & $8 \cdot 3^{-9}$ \\ [0.2ex]
                     & $3$      & $4$ & $9$ &                                            & $3$ & $3^5$ & $4 \cdot 3^{-9}$ \\ [0.2ex]
  $\mathrm{IV}^*(4)$ & $3$ & $3$ & $10$ && $4$ & $3^5$ & $4 \cdot 3^{-9}$ \\ [0.2ex]
  $\mathrm{IV}^*(5)$ & $\geqslant 4$ & $4$ & $11$ && $5$ & $3^5$ & $2 \cdot 3^{-9}$ \\ [0.2ex]
  $\mathrm{III}^*$   & $\geqslant 3$ & $3$ & $9$ & \eqref{eq:p3_type_IIIstar_extra_condition} & $2$ & $3^5$ & $4 \cdot 3^{-9}$ \\ [0.2ex]
                     & $3$ & $\geqslant 5$ & $9$ &                                            & $2$ & $3^5$ & $2 \cdot 3^{-9}$ \\ [0.2ex]
  $\mathrm{II}^*(3)$ & $3$ & $3$ & $11$ && $3$ & $3^6$ & $4 \cdot 3^{-10}$ \\ [0.2ex]
  $\mathrm{II}^*(4)$ & $4$ & $5$ & $12$ && $4$ & $3^6$ & $4 \cdot 3^{-11}$ \\ [0.2ex]
  $\mathrm{II}^*(5)$ & $\geqslant 5$ & $5$ & $13$ && $5$ & $3^6$ & $2 \cdot 3^{-11}$ \\ [0.2ex]
  $3^{12}\mathrm{I}_0$ & $3$ & $3$ & $12$ && $0$ & $3^7$ & $4 \cdot 3^{-11}$ \\ [0.2ex]
  $3^{12}\mathrm{I}_\m,\; \m \geqslant 1$ & $3$ & $3$ & $\m+12$ && $1$ & $3^{\m+7}$ & $4 \cdot 3^{-\m-11}$ \\ [0.2ex]
  $3^4 \!\mid\! a \;\text{and}\; 3^6 \!\mid\! b$ & $\geqslant 4$ & $\geqslant 6$ & $\geqslant 12$ && --- & $3^6$ & $3^{-10}$ \\ [0.2ex]
  \hline
\end{tabular}
\begin{tablecap}\label{table:red3_all}
  Reduction information at $p = 3$ for the family $\cF$. See the surrounding text for further details.
\end{tablecap}
\vspace{-\baselineskip}
\end{table}

The information presented in \cref{table:redp,table:red2_all,table:red3_all} either comes directly from, or is easily derived from \cite{papadopoulos}. For $p\geq 5$, similar information is presented in \cite[Thm.\ 1.6]{SSW} and \cite[Prop.\ 2.2]{cremona_sadek}. In the case $p = 2$ we made some minor adjustments to \cite{papadopoulos}, described in \cref{rem:papadopoulos_correction}.

\begin{remark}
  \label{rem:papadopoulos_correction}
  The following comments pertain to \cref{table:red2_all}.
\begin{enumerate}[label=(\roman*)]
\item
  There is a minor error in \cite[Tableau IV]{papadopoulos}. The condition for the second case of ``Equation non minimale'' is given as $v_2(c_4) \geqslant 8$, $v_2(c_6) = 9$, and $v_2(\Delta) = 12$. The condition on $\Delta$ should read $v_2(\Delta) \geqslant 12$.
  An example is $E': y^2 + xy + y = x^3 + 8x^2 + 588x + 1724$, which is the curve one obtains heading into step 11 of Tate's algorithm as described in \cite[\S IV.9]{silverman} upon inputting $E: y^2 = x^3 + 565x + 6$.
\item
  The reduction types we call $\mathrm{I}_2^*(6)$ and $\mathrm{I}_3^*(5)$ cannot be distinguished by $v_2(c_4)$, $v_2(c_6)$, and $v_2(\Delta)$ alone, and none of the conditions from \cite[\S II.3]{papadopoulos} address this. We have introduced \eqref{eq:p2_extra_condition_5} which distinguishes them in the case of short Weierstrass models.
\item
  We modified the condition $v_2(c_6) \geqslant 9$ to $v_2(c_6) \geqslant 10 \Leftrightarrow v_2(b) \geqslant 5$. For short Weierstrass equations, if $v_2(a) = 2$ and $v_2(b) = 4$ then necessarily $v_2(\Delta) \geqslant 13$, whereas $v_2(\Delta) = 12$ for both $\mathrm{I}_2^*(6)$ and $\mathrm{I}_3^*(5)$.
\end{enumerate}
\end{remark}

\Cref{fig:red23} illustrates, for $p = 2$ and $3$, the sets $S_\T$, i.e.\ the conditions on $a$ and $b$ which determine $E_{a,b}$ to have reduction type $\T$.

\begin{figure}[H]
  \hspace{-1.6em}
  \includegraphics[width=1.02\textwidth]{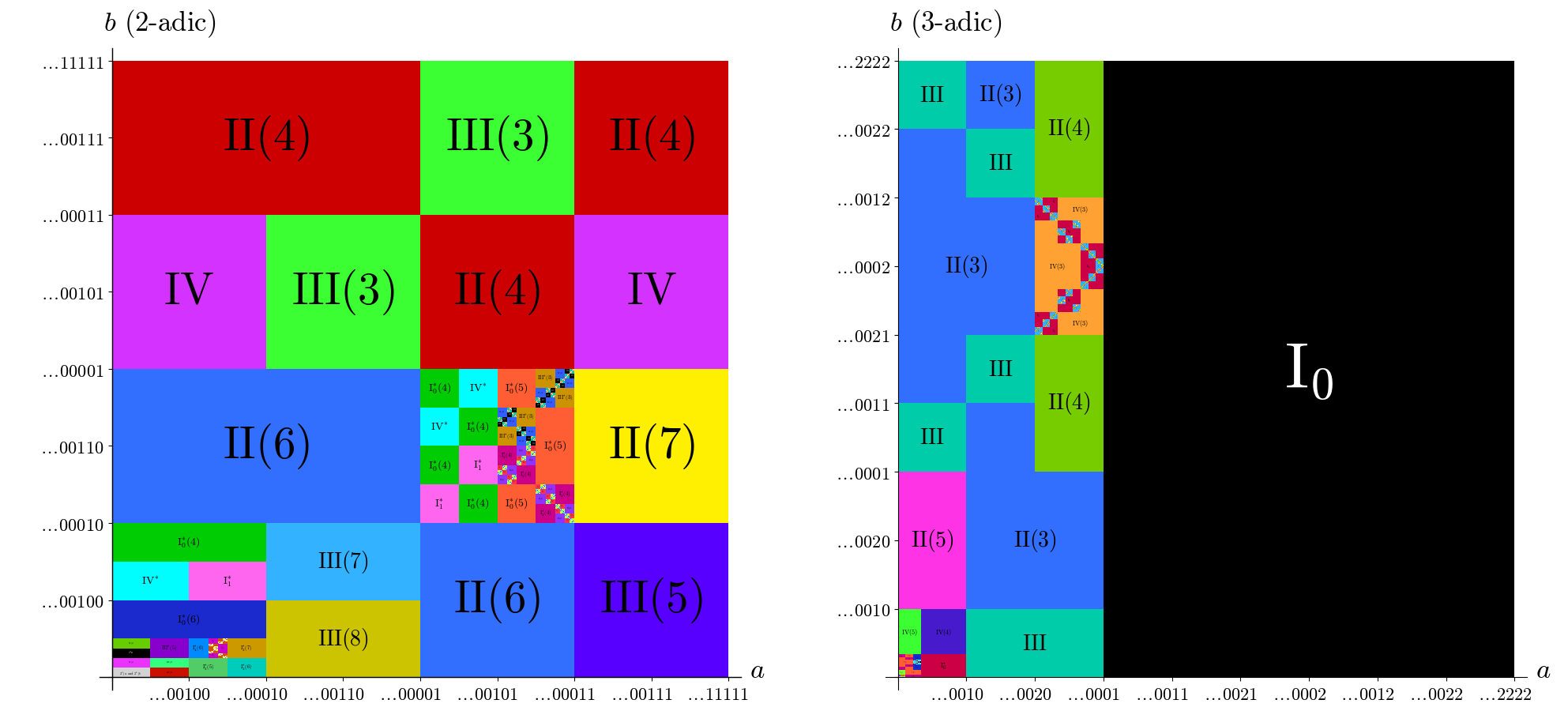}
  \begin{figurecap}\label{fig:red23}
    Reduction types of $E_{a,b}: y^2 = x^3 + ax + b$ at $p = 2$ (left) and $p = 3$ (right), as a function of $a$ and $b$. The axes order the coefficients $a$ and $b$ $p$-adically, with coefficients more highly divisible plotted nearer to the origin.
    For example, at $p = 2$, $E_{a,b}$ has reduction type $\mathrm{II}(7)$ (i.e.\ reduction type $\mathrm{II}$ with the additional condition that $v_2(N) = 7$) if $a = 3 \mod 4$ and $2 \parallel b$.
    A higher resolution version of this figure is available at \cite{github_conductors}: \href{https://github.com/thealexcowan/conductordistribution/blob/main/reduction_map_at_2.png}{$p = 2$}, \href{https://github.com/thealexcowan/conductordistribution/blob/main/reduction_map_at_3.png}{$p = 3$}. The illustrated conditions on $a$ and $b$ are also listed in \cref{table:red2_all,table:red3_all}.
  \end{figurecap}
\end{figure}

At the primes $p = 2$ and $3$, the behaviour of the sort tabulated above differs between the families $\cF$ and $\cFY$. We now turn our attention to the latter family.

\begin{lemma}\label{local_densities_3}
  Each $E_{a,b} \in \cFY(H)$ has good reduction at $3$.
\end{lemma}
\begin{proof}
  \Cref{fig:red23} or \cite[\S 2.1]{young:elliptic_curves}.
\end{proof}

\begin{lemma}\label{local_densities_2}
  \Cref{table:red2} gives the reduction type at $p = 2$ of $E_{a,b} \in \cFY(H)$, where $r$ and $t$ are as in \cref{cF_tilde_def}.
\begin{table}[H]
\emph{
\begin{tabular}{|c|c|c|l||c|c|c|l|}
  \hline
  $r$, $t$ & $a \mod 12$ & $b \mod 12$ & Reduction type at $2$ & $r$, $t$ & $a \mod 12$ & $b \mod 12$ & Reduction type at $2$ \rule{0pt}{1em}\\
  [0.2ex]
  \hline
  \multirow{4}{*}{$1,1$} & $1$ & $1$ & Type II  & \multirow{4}{*}{$2,1$} & $2$ & $1$ & Type III \rule{0pt}{1em}\\ [0.2ex]
  & $1$ & $7$ & Type III & & $2$ & $7$ & Type II  \\ [0.2ex]
  & $7$ & $1$ & Type IV  & & $8$ & $1$ & Type IV  \\ [0.2ex]
  & $7$ & $7$ & Type II  & & $8$ & $7$ & Type II  \\ [0.2ex]
  \hline
  \multirow{4}{*}{$1,3$} & $1$ & $3$ & Type III & \multirow{4}{*}{$2,3$} & $2$ & $3$ & Type II  \rule{0pt}{1em}\\ [0.2ex]
  & $1$ & $9$ & Type II  & & $2$ & $9$ & Type III \\ [0.2ex]
  & $7$ & $3$ & Type II  & & $8$ & $3$ & Type II  \\ [0.2ex]
  & $7$ & $9$ & Type IV  & & $8$ & $9$ & Type IV  \\ [0.2ex]
  \hline
  \multirow{4}{*}{$1,5$} & $1$ &  $5$ & Type II  & \multirow{4}{*}{$2,5$}  & $2$ &  $5$ & Type III \rule{0pt}{1em}\\ [0.2ex]
  & $1$ & $11$ & Type III & & $2$ & $11$ & Type II  \\ [0.2ex]
  & $7$ &  $5$ & Type IV  & & $8$ &  $5$ & Type IV  \\ [0.2ex]
  & $7$ & $11$ & Type II  & & $8$ & $11$ & Type II  \\ [0.2ex]
  \hline
  $5,t$ & \multicolumn{3}{c||}{Same as $1,t$ with $a,b \mapsto a+4,b$} & $4,t$ & \multicolumn{3}{c|}{Same as $2,t$ with $a,b \mapsto a+8,b$} \rule{0pt}{1em}\\ [0.2ex]
  \hline
\end{tabular}
}
\begin{tablecap}\label{table:red2}
  Reduction type at $p = 2$ of $E_{a,b} \in \cFY(H)$ as a function of $a,b \ \mathrm{mod}\ 12$, with corresponding choices of $r,t$ from the definition of $\cFY(H)$.
\end{tablecap}
\vspace{-\baselineskip}
\end{table}
\end{lemma}
\begin{proof}
  This information can be read directly from \cref{fig:red23}.
  Alternatively, follow \cite[\S IV.9]{silverman}. In step 2, to put the singular point at $(0,0)$, when $2 \mid r$ make the change of variables $y \mapsto y+1$, and when $2 \nmid r$ make the changes of variables $y \mapsto y+x$ then $x \mapsto x+1$. The Weierstrass equation $y^2 = x^3 + ax + b$ becomes in each of these cases $$y^2 + 2y = x^3 + ax + b-1 \quad\text{or}\quad y^2 + 2xy + 2y = x^3 + 2x^2 + (a+1)x + a+b$$ respectively. Following the algorithm and using the assumptions on $a$ and $b$, we obtain \cref{table:red2}. Note that in our case the algorithm terminates in step 5 at the latest.
\end{proof}

It is immediate from \cref{table:redp,table:red2_all,table:red3_all} that, for each prime $p$ and each non-negative integer $k$, there exists a corresponding positive integer $\QM$ such that, for $E_{a,b} \in \cF(H)$, the condition $\gcd(\Delta/N, p^\infty) = p^k$ is equivalent to a congruence condition on $(a,b)$ modulo $\QM$. The values of $\rho(p,\m)$ from \cref{rho_def} are chosen such that, for each $k \in \Z_{\geqslant 0}$, the proportion of residue classes satisfying said congruence condition is as given in \cref{table:rho_def}. The same holds true for $\rhoY(p,\m)$ from \cref{rho_tilde_def} vis-\`a-vis \cref{local_densities_3,local_densities_2}. For $p \geqslant 5$ this information appears in \cite[Table 2]{SSW}.

\begin{definition}\label{def:QC}
For $\q > 5$, define
\begin{align*}
  \QD \coloneqq \prod_{\phantom{{}_*}p \,<\, \q} p^{\left\lfloor \frac{\log \q}{\log p} \right\rfloor + 2}
  \quad\quad\text{and}\quad\quad
  C \coloneqq \prod_{\phantom{{}_*}p \,<\, \q} p^{\left\lfloor \frac{\log \q}{\log p} \right\rfloor}.
\end{align*}
For each $\m \mid C$, define $S_{\QD,\m}$ to be the subset of $(\Z/\QD\Z)^2$ such that $(a,b) \in S_{\QD,\m}$ if and only if
\begin{align}
  \label{eq:d4d6_m} &\gcd\!\left(\frac{|\Delta_{E_{a,b}}|}{N_{E_{a,b}}}, C\right) = \m
  \shortintertext{and}
  \label{eq:d4d6_minimal} &\text{for all $p < \q$, if $p^6 \mid \QD$ and $p^4 \mid a$ then $p^6 \nmid b$}.
\end{align}
In the condition \eqref{eq:d4d6_m} we write $E_{a,b}$ for $(a,b) \in (\Z/\QD\Z)^2$ to mean any elliptic curve $E_{\hat{a},\hat{b}} \in \cF(H)$ with $(\hat{a},\hat{b}) \in \Z^2$ such that $\hat{a} = a \mod \QD$ and $\hat{b} = b \mod \QD$. \Cref{table:redp,table:red2_all,table:red3_all} confirm that the quantity appearing in \eqref{eq:d4d6_m} does not depend on the choice of lifts $\hat{a}$ and $\hat{b}$.

Moreover, define $\SY_{\QD,\m}$ to be the subset of $(\Z/\QD\Z)^2$ satisfying \eqref{eq:d4d6_m}, \eqref{eq:d4d6_minimal}, as well as
\begin{align}
  \label{eq:d4d6_ab} &a = r \mod 6,\,\, b = t \mod 6, \,\,\text{where $r,t$ are as in \cref{cF_tilde_def}.}
\end{align}
\end{definition}

The following lemma establishes that the hypotheses of \cref{disc_dist} are satisfied for $S_{\QD,\m}$ and $\SY_{\QD,\m}$ as defined above.

\begin{lemma}\label{d4d6_by_type}
  With the notation of \cref{def:QC}, for all $\q > 5$, all $\m\mid C$, and all $(a,b) \in S_{\QD,\m}$ or $\SY_{\QD,\m}$, there does not exist $d \in \Z$ such that $\gcd(d,\QD) > 1$ and $\gcd(d^4,\QD) \mid \gcd(a,\QD)$ and $\gcd(d^6,\QD) \mid \gcd(b,\QD)$.
\end{lemma}
\begin{proof} We analyze each possible prime factor $p$ of $d$ individually.
  
  If $p^6 \mid \QD$, then the claim of \cref{d4d6_by_type} is automatic from \eqref{eq:d4d6_minimal}.

  If $p^6 \nmid \QD$ then $p^4 \nmid C$ by construction. Examining \cref{table:redp,table:red2_all,table:red3_all}, we see that we need only consider minimal $E_{a,b}$ with reduction type $\mathrm{I}_0$, $\mathrm{I}_1$, $\mathrm{I}_2$, $\mathrm{I}_3$, $\mathrm{I}_4$, $\mathrm{II}$, $\mathrm{III}$, or $\mathrm{IV}$. Upon noting that the standard invariants $c_4$ and $c_6$ \cite[(5) and (6)]{papadopoulos} of the short Weierstrass equation $E_{a,b}$ are $-2^43a$ and $-2^53^3b$ respectively, tableaux I, II, et IV of \cite{papadopoulos} show that, for each of the aforementioned reduction types, either $p^2 \nmid a$ or $p^3 \nmid b$. (Alternatively, see the proof of \cite[Thm.\ 1.6]{SSW}.) However, $p \mid \QD$ implies $p^3 \mid \QD$ by definition, so if $p \mid d$ then $p^3 \mid \gcd(d^4, \QD)$ and hence either $\gcd(d^4, \QD) \nmid \gcd(a,\QD)$ or $\gcd(d^6, \QD) \nmid \gcd(b,\QD)$.
\end{proof}


\section{Approximations and tail bounds}\label{sec:tail}

The proofs in \cref{sec:proofs} make use of several technical results for approximating various quantities. These technical results are catalogued here.

\begin{lemma}\label{euler_prod_approx}
  \begin{align*}
    \prod_{\phantom{{}_*}p < \q}\frac{1 - p^{-2}}{1 - p^{-10}} = \frac{\zeta(10)}{\zeta(2)}(1 + O(\q^{-1})).
  \end{align*}
\end{lemma}
\begin{proof}
  We compute
  \begin{align*}
    \log \prod_{\phantom{{}_*}p < \q}\frac{1 - p^{-2}}{1 - p^{-10}}
    &= \sum_{p} \log(1 - p^{-2}) - \log(1 - p^{-10}) - \sum_{p \geqslant \q} \log(1 - p^{-2}) - \log(1 - p^{-10})\\
    &= \log\frac{\zeta(10)}{\zeta(2)} + O\!\left(\sum_{p \geqslant \q} p^{-2} + p^{-10}\right) \hspace{3cm}\text{(using $\log(1+\eps) = O(\eps)$)}\\
    &= \log\frac{\zeta(10)}{\zeta(2)} + O(\q^{-1}).
  \end{align*}
  The result then follows from the fact that $\exp(\eps) = 1 + O(\eps)$.
\end{proof}

\begin{lemma}\label{rad_comparison}
  \begin{align*}
    \prod_p \frac{\rho(p,\m)}{1 - p^{-2}} \ll \frac{1}{\m}\prod_{p \mid \m} \frac{3}{p}.
  \end{align*}
\end{lemma}
\begin{proof}
  Follows from inspecting \cref{table:rho_def} case by case.
\end{proof}

\begin{lemma}\label{rad_euler_product}
  For $\Re(s) > 0$,
  \begin{align*}
    \sum_{\m = 1}^\infty \frac{1}{\m^{s}} \prod_{p\mid \m}\frac{3}{p} = \prod_p \left(1 + \frac{3}{p(p^s - 1)}\right).
  \end{align*}
\end{lemma}
\begin{proof}
  Follows from a straightforward Euler product calculation.
\end{proof}

\begin{lemma}\label{db_sum}
  \begin{align*}
    \sum_{\m = 1}^\M \prod_{p\mid \m}\frac{3}{p} \ll \M^\eps.
  \end{align*}
\end{lemma}
\begin{proof}
  \begin{align*}
    \sum_{\m = 1}^\M \frac{1}{\M^\eps}\prod_{p\mid \m}\frac{3}{p} < \sum_{\m = 1}^\M \frac{1}{\m^\eps}\prod_{p\mid \m}\frac{3}{p},
  \end{align*}
  and from \cref{rad_euler_product} we see that the series on the right converges as $\M \to \infty$.
\end{proof}

\begin{lemma}\label{FDelta_bound}
  For $\lambda > 0$,
  \begin{align*}
    F_\Delta(\lambda) - F_\Delta(-\lambda) \ll \lambda^{\frac{5}{6}}.
  \end{align*}
\end{lemma}
\begin{proof}
  The quantity on the left is between $0$ and $1$, so the statement is automatic for $\lambda$ bounded away from $0$. For $\lambda \to 0$,
  \begingroup
  \allowdisplaybreaks  
  \begin{align*}
    \frac{1}{4}\int_{-1}^1\int_{-1}^1 &\begin{cases}1 & \text{if $-\lambda < -16(4\alpha^3 + 27\beta^2) < \lambda$} \\ 0 & \text{otherwise}\end{cases} \,d\alpha \,d\beta\\
    &= \frac{1}{12}\int_{0}^{1}\int_{-1}^{1} \begin{cases}1 & \text{if $-\lambda < -16(4u + 27v) < \lambda$} \\ 0 & \text{otherwise}\end{cases} \,\frac{du}{u^{\frac{2}{3}}} \,\frac{dv}{v^{\frac{1}{2}}} \hspace{2cm}\text{(setting $u \coloneqq \alpha^3$, $v \coloneqq \beta^2$)}\\
    &= \frac{1}{12}\int_{0}^{\frac{4}{27}\left(1 - \frac{\lambda}{64}\right)}\int_{-\frac{27}{4}v - \frac{\lambda}{64}}^{-\frac{27}{4}v + \frac{\lambda}{64}} \,\frac{du}{u^{\frac{2}{3}}} \,\frac{dv}{v^{\frac{1}{2}}}
    + \frac{1}{12}\int_{\frac{4}{27}\left(1 - \frac{\lambda}{64}\right)}^{\frac{4}{27}\left(1 + \frac{\lambda}{64}\right)}\int_{-1}^{-\frac{27}{4}v + \frac{\lambda}{64}} \,\frac{du}{u^{\frac{2}{3}}} \,\frac{dv}{v^{\frac{1}{2}}}\\
    &= \frac{1}{4}\int_{0}^{\frac{4}{27}\left(1 - \frac{\lambda}{64}\right)} \left(-\frac{27}{4}v + \frac{\lambda}{64}\right)^{\frac{1}{3}} - \left(-\frac{27}{4}v - \frac{\lambda}{64}\right)^{\frac{1}{3}} \,\frac{dv}{v^{\frac{1}{2}}} + O(\lambda)\\
    &= \frac{1}{4}\left(\frac{\lambda}{64}\right)^{\frac{1}{3}}\int_{0}^{\frac{\lambda}{432}} \left(1 - \frac{432v}{\lambda}\right)^{\frac{1}{3}} + \left(1 + \frac{432v}{\lambda}\right)^{\frac{1}{3}} \,\frac{dv}{v^{\frac{1}{2}}}\\
    &\quad
    + \frac{1}{4}\int_{\frac{\lambda}{432}}^{\frac{4}{27}\left(1 - \frac{\lambda}{64}\right)} \left(\frac{27}{4} + \frac{\lambda}{64v}\right)^{\frac{1}{3}} - \left(\frac{27}{4} - \frac{\lambda}{64v}\right)^{\frac{1}{3}} \,\frac{dv}{v^{\frac{1}{6}}} + O(\lambda)\\
    &= \frac{1}{4}\left(\frac{\lambda}{64}\right)^{\frac{1}{3}}\int_{0}^{\frac{\lambda}{432}} \left(1 - \frac{432v}{\lambda}\right)^{\frac{1}{3}} + \left(1 + \frac{432v}{\lambda}\right)^{\frac{1}{3}} \,\frac{dv}{v^{\frac{1}{2}}}\\
    &\quad+ \frac{1}{4}\int_{\frac{\lambda}{432}}^{\frac{4}{27}\left(1 - \frac{\lambda}{64}\right)} \frac{\left(\frac{27}{4} + \frac{\lambda}{64v}\right) - \left(\frac{27}{4} - \frac{\lambda}{64v}\right)}{\left(\frac{27}{4} + \frac{\lambda}{64v}\right)^{\frac{2}{3}} + \left(\frac{27}{4} + \frac{\lambda}{64v}\right)^{\frac{1}{3}} \left(\frac{27}{4} - \frac{\lambda}{64v}\right)^{\frac{1}{3}} + \left(\frac{27}{4} - \frac{\lambda}{64v}\right)^{\frac{2}{3}}} \,\frac{dv}{v^{\frac{1}{6}}} + O(\lambda)\\
    &\leqslant \frac{\lambda^{\frac{1}{3}}}{8} (2^{\frac{1}{3}} + 1) \left(\frac{\lambda}{432}\right)^{\frac{1}{2}} + \frac{\lambda}{128}\int_{\frac{\lambda}{432}}^{\frac{4}{27}\left(1 - \frac{\lambda}{64}\right)} \frac{1}{\left(\frac{27}{4}\right)^{\frac{2}{3}} + 0 + 0} \,\frac{dv}{v^{\frac{7}{6}}} + O(\lambda)\\
    &\ll \lambda^{\frac{5}{6}} + \lambda (\lambda^{-\frac{1}{6}} + 1) + \lambda\\
    &
    \ll \lambda^{\frac{5}{6}}. \qedhere
  \end{align*}
  \endgroup
\end{proof}

\begin{lemma}\label{FN_small_bound}
  For $\lambda > 0$,
  \begin{align*}
    \sum_{\m = 1}^\infty \big(F_\Delta(\m\lambda) - F_\Delta(-\m\lambda)\big)\prod_p \frac{\rho(p,\m)}{1 - p^{-2}} \ll \lambda^{\frac{5}{6}}.
  \end{align*}
\end{lemma}
\begin{proof}
  Using \cref{rad_comparison} and \cref{FDelta_bound},
  \begin{align*}
    \sum_{\m = 1}^\infty \big(F_\Delta(\m\lambda) - F_\Delta(-\m\lambda)\big)\prod_p \frac{\rho(p,\m)}{1 - p^{-2}} \ll \sum_{\m = 1}^\infty \lambda^{\frac{5}{6}} \frac{1}{\m^{\frac{1}{6}}}\prod_{p \mid \m} \frac{3}{p}.
  \end{align*}
  From \cref{rad_euler_product} we see that the series on the right converges.
\end{proof}

\begin{remark}
  \label{rem:diff}
  With some calculus one finds that $F_\Delta(\lambda)$ is not twice differentiable at $\lambda = -432$.
  This is essentially because that's the point at which $-16(4\alpha^3 + 27\beta^2) \leqslant \lambda$ forces $\alpha \geqslant 0$, and the function $x^{\frac{1}{3}}$ is not differentiable at $0$. It follows that the right hand sides of \cref{thm:conductor_distribution_FN} are not twice differentiable at infinitely many points.
  The function $F_\Delta$ is also not twice differentiable at $-64$. These kinks are visible in \cref{fig:thm1}.
\end{remark}

\begin{proposition}\label{FN_bigm}
  Let $0 < \delta < \tfrac{1}{6}$,
  \begin{align*}
    \frac{\#\!\left\{E \in \cF(H) \,:\, H^{-\delta} < \frac{N_E}{H} ,\,\, \M < \frac{|\Delta_E|}{N_E}\right\}}{\#\cF(H)} \ll \M^{-1+\eps}
    ,
  \end{align*}
  and idem for $\cFY(H)$.
\end{proposition}
\begin{proof}
  In the proof below, all instances of $\cF$ may be replaced by $\cFY$.
  
  Note that $-496H \leqslant \Delta_E \leqslant 64H$.
  \begin{align*}
    \frac{\#\!\left\{E \in \cF(H) \,:\, H^{-\delta} < \frac{N_E}{H},\,\, \M < \frac{|\Delta_E|}{N_E}\right\}}{\#\cF(H)}
    &= \sum_{\M \,<\, |\m| \,<\, 496 H^\delta} \frac{\#\left\{ E \in \cF(H) \,:\, \frac{\Delta_E}{N_E} = \m,\,\, H^{-\delta} < \frac{\Delta_E}{\m H} \right\}}{\#\cF(H)}\\
    &\leqslant \sum_{\M \,<\, |\m| \,<\, 496 H^\delta} \frac{\#\left\{ E \in \cF(H) \,:\, \m \mid \frac{\Delta_E}{N_E},\,\, H^{-\delta} < \frac{\Delta_E}{\m H} \right\}}{\#\cF(H)}.
  \end{align*}
  Define $Q_\m \coloneqq |\m| \prod_{p\mid \m} p^2$. By \cref{table:redp,table:red2_all,table:red3_all}, the condition $\m \mid \tfrac{\Delta_E}{N_E}$ is determined by a congruence mod $Q_\m$. Applying \cref{disc_dist,rad_comparison,FH_size}
  then gives
  \begin{align*}
    \frac{\#\!\left\{E \in \cF(H) \,:\, H^{-\delta} < \frac{N_E}{H},\,\, \M < \frac{|\Delta_E|}{N_E}\right\}}{\#\cF(H)}
    &\ll \sum_{\M \,<\, \m \,<\, 496 H^\delta} \left[ \frac{1}{\m}\prod_{p \mid \m} \frac{3}{p} + O\!\left(H^{-\frac{1}{3}} \prod_{p\mid \m} p + H^{-\frac{5}{6}} \m \prod_{p\mid \m} p^3\right) \right]\\
    &\ll \sum_{\M \,<\, \m \,<\, 496 H^\delta} \left[ \frac{1}{\m}\prod_{p \mid \m} \frac{3}{p} + O\!\left(H^{-\frac{1}{3}} \m + H^{-\frac{5}{6}} \m^4\right) \right]\\
    &\ll \sum_{\M \,<\, \m \,<\, 496 H^\delta} \left[ \frac{1}{\m}\prod_{p \mid \m} \frac{3}{p} \right] + O\!\left(H^{-\frac{1}{3} + 2\delta} + H^{-\frac{5}{6} + 5\delta}\right).
  \end{align*}

  Let $\xi > 1$ be a value to be chosen later. A $\xi$-adic decomposition of the sum above gives
  \begin{align*}
    \sum_{\M \,<\, \m \,<\, 496 H^\delta} \frac{1}{\m}\prod_{p \mid \m} \frac{3}{p}
    &\ll \sum_{0 \,<\, k \,<\, \log_\xi 496 H^\delta - \log_\xi \M} \;\;\sum_{\M \xi^k \,<\, \m \,<\, \M \xi^{k+1}} \frac{1}{\m}\prod_{p \mid \m} \frac{3}{p}\\
    &\ll \sum_{0 \,<\, k \,<\, \log_\xi 496 H^\delta - \log_\xi \M} \;\frac{1}{\M \xi^k}\;\sum_{\M \xi^k \,<\, \m \,<\, \M \xi^{k+1}} \prod_{p \mid \m} \frac{3}{p}\\
    &\ll \sum_{0 \,<\, k \,<\, \log_\xi 496 H^\delta - \log_\xi \M} \;\frac{1}{\M \xi^k} (\m \xi^{k+1})^\eps \hspace{2.5cm} (\text{using \cref{db_sum}})\\
    &\ll \frac{(\m \xi)^\eps}{\m} \sum_{0 \,<\, k \,<\, \log_\xi 496 H^\delta - \log_\xi \M} (\xi^{-1+\eps})^k\\
    &\ll \frac{(\m \xi)^\eps}{\m} \frac{1 - \left(\frac{496 H^\delta}{\m}\right)^{-1 + \eps}}{1 - \xi^{-1+\eps}}.
  \end{align*}
  If $\M > 496H^\delta$ then \cref{FN_bigm} holds because the quantity on the left hand side is $0$. Otherwise, taking e.g.\ $\xi = 2$ completes the proof.
\end{proof}

\begin{lemma}\label{N_56_bound}
  \begin{align*}
    \#\{ E \in \cF(H) \,:\, N_E < H^{\frac{5}{6}+\eps},\,\, |\Delta_E| > H^{1-\eps}\} \ll H^{\frac{19}{24} + \eps}
    ,
  \end{align*}
  and idem for $\cFY(H)$.
\end{lemma}
\begin{proof}
  It is sufficient to prove the claim for $\cFQ$, since $\cFY(H) \subseteq \cFQ(H)$ for all $H$.
  
  Let $\cFQ(H)^{\mathrm{min}} \coloneqq \{E_{a,b} \in \cF(H) \,:\, p^2 \mid a \Rightarrow p^3 \nmid b\}$ be the set of quadratic-twist-minimal curves of $\cFQ(H)$. The quadratic twist $E^{(d)}$ of the elliptic curve $E$ by $d$ has $$N_{E^{(d)}} = d^2 N_E \quad\quad\text{and}\quad\quad \Delta_{E^{(d)}} = d^6 \Delta_E.$$ If $E$ is quadratic-twist-minimal and $d$ is such that $N_{E^{(d)}} < H^{\frac{5}{6}+\eps}$ and $H^{1-\eps} < |\Delta_{E^{(d)}}| < 496H$, this imposes the constraints
  \begin{align*}
    d < \frac{H^{\frac{5}{12}+\eps}}{N_E^\half} \quad\quad\text{and}\quad\quad \frac{H^{\frac{1}{6}-\eps}}{N_E^{\frac{1}{6}} \sqrt{\frac{|\Delta_E|}{N_E}}^{\frac{1}{3}}} < d < \frac{496^{\frac{1}{6}}H^{\frac{1}{6}}}{N_E^{\frac{1}{6}}\sqrt{\frac{|\Delta_E|}{N_E}}^{\frac{1}{3}}}.
  \end{align*}
  These imply that the number of integers $d$ satisfying these constraints is $\ll H^\eps$, and that $\sqrt{\frac{|\Delta_E|}{N_E}} > \frac{N_E}{H^{\frac{3}{4}+\eps}}.$
  Thus,
  \begin{align*}
    \#\{ E \in \cFQ(H) \,:\, N_E < H^{\frac{5}{6}+\eps},\,\, |\Delta_E| > H^{1-\eps}\}
    \ll H^\eps \#\!\left\{E \in \cFQ(H)^{\mathrm{min}} \,:\, \sqrt{\frac{|\Delta_E|}{N_E}} > \frac{N_E}{H^{\frac{3}{4}+\eps}}\right\}.
  \end{align*}
  By \cite[Prop.\ 1]{duke_kowalski},
  \begin{align*}
    \#\!\left\{E \in \cFQ(H)^{\mathrm{min}} \,:\, \sqrt{\frac{|\Delta_E|}{N_E}} > \frac{N_E}{H^{\frac{3}{4}+\eps}},\,\, N_E \leqslant H^{\frac{19}{24}}\right\} \ll H^{\frac{19}{24}+\eps}.
  \end{align*}
  Note from \cite[Table 1]{SSW} that the quantity $Q(f)$ appearing in \cite[Thm.\ 5.1]{SSW} is $\gg \sqrt{\frac{|\Delta_E|}{N_E}}$. Thus,
  \begin{align*}
    \#\!\left\{E \in \cFQ(H)^{\mathrm{min}} \,:\, \sqrt{\frac{|\Delta_E|}{N_E}} > \frac{N_E}{H^{\frac{3}{4}+\eps}},\,\, N_E > H^{\frac{19}{24}}\right\}
    &< \#\!\left\{E \in \cFQ(H)^{\mathrm{min}} \,:\, \sqrt{\frac{|\Delta_E|}{N_E}} > H^{\frac{1}{24}}\right\}\\
    &\ll H^{\frac{19}{24}+\eps}.\qedhere
  \end{align*}
\end{proof}


\section{Proofs of the main theorems}\label{sec:proofs}
\begin{proof}[\hypertarget{proof:thm1}{Proof of \cref{thm:conductor_distribution_FN}}]
We will prove that, for any positive real numbers $\lambda_0 < \lambda_1$ and any $\q \geqslant \tfrac{496}{\lambda_0}$,
\begin{align}
  &\frac{\#\!\left\{E \in \cF(H) \,:\, \lambda_0 < \frac{N_E}{H} < \lambda_1\right\}}{\#\cF(H)}\nonumber\\
  &\hspace{3cm}= \prod_{\phantom{{}_*}p < \q}\frac{1 - p^{-2}}{1 - p^{-10}} \cdot \sum_{\m \neq 0} \left(\big|F_\Delta(\m\lambda_1) - F_\Delta(\m\lambda_0)\big| \prod_p \frac{\rho(p,\m)}{1 - p^{-2}}\right) \label{eq:conductor_distribution}\\
  &\hspace{3.5cm}+ O\!\left(\q H^{-\frac{1}{3}+\frac{3q}{\log H}} + \q^{-1+\eps}\right).\nonumber 
\end{align}
\Cref{thm:conductor_distribution_FN} for the family $\cFQ(H)$ will then follow from \cref{euler_prod_approx} and setting $\q = (\tfrac{1}{9} - \eps)\log H$ in the above. The proof of \cref{thm:conductor_distribution_FN} for the family $\cFY(H)$ is the same, mutatis mutandis.

For every fixed $\lambda_0$, the sum over $\m$ in the right hand side above contains only finitely many nonzero terms: $-496 < \m\lambda_0 < 64$. Moreover, because $\rho(p,\m) = 1 - p^{-2}$ when $p \nmid \m$ and $p \neq 2$ (see \cref{rho_def}), the product over $p$ in \eqref{eq:conductor_distribution} is also finite.
  
Fix positive real numbers $\lambda_0 < \lambda_1$. Then
\begin{align*}
  \#\left\{ E \in \cF(H) \,:\, \lambda_0 < \frac{N_E}{H} < \lambda_1 \right\}
  = \sum_{\m \,\in\, \Z_{\neq 0}} \#\left\{ E \in \cF(H) \,:\, \frac{\Delta_E}{N_E} = \m,\,\, \lambda_0 < \frac{\Delta_E}{\m H} < \lambda_1 \right\}.
\end{align*}
Since $N_E \geqslant 1$ and $-496H < \Delta_E < 64H$, only finitely many terms in the sum above will be nonzero:
\begin{align*}
  \#\left\{ E \in \cF(H) \,:\, \lambda_0 < \frac{N_E}{H} < \lambda_1 \right\}
  = \sum_{\substack{-\frac{496}{\lambda_0} \,<\, \m \,<\, \frac{64}{\lambda_0} \\ \m \,\neq\, 0}} \#\left\{ E \in \cF(H) \,:\, \frac{\Delta_E}{N_E} = \m,\,\, \lambda_0 < \frac{\Delta_E}{\m H} < \lambda_1 \right\}.
\end{align*}

Let $\q \geqslant \tfrac{496}{\lambda_0}$, and let $\QD$, $C$, and $S_{\QD,\m}$ for $\m \mid C$ be as in \cref{def:QC}. 

Every $(a,b)$ such that $E_{a,b} \in \cF(H)$ automatically satisfies 
\eqref{eq:d4d6_minimal} (in the case of the family $\cFY(H)$, both \eqref{eq:d4d6_minimal} and \eqref{eq:d4d6_ab}). If $0 < \m < \q$ and $(a\mod \QD,\, b\mod \QD) \in S_{\QD,\m}$ and $\frac{|\Delta_E|}{N_E} \neq \m$, then by construction
$\frac{|\Delta_E|}{N_E} > \q$. 
Also by construction, if $(a\mod \QD,\, b\mod \QD) \in S_{\QD,\m}$, then $(a\mod \QD,\, b\mod \QD) \not\in S_{\QD,\m'}$ for any $0 < \m' < \q$, $\m' \neq \m$. Thus,
\begin{align}
  \nonumber
  &\sum_{\substack{-\frac{496}{\lambda_0} \,<\, \m \,<\, \frac{64}{\lambda_0} \\ \m \,\neq\, 0}} \#\left\{ E \in \cF(H) \,:\, \frac{\Delta_E}{N_E} = \m,\,\, \lambda_0 < \frac{\Delta_E}{\m H} < \lambda_1 \right\}\\
  \label{eq:SQm_decomp}
  &\hspace{2cm}= \sum_{\substack{-\frac{496}{\lambda_0} \,<\, \m \,<\, \frac{64}{\lambda_0} \\ \m \,\neq\, 0}} \#\left\{ E_{a,b} \in \cF(H) \,:\, (a\mod \QD,\, b\mod \QD) \in S_{\QD,\m},\,\, \lambda_0 < \frac{\Delta_E}{\m H} < \lambda_1 \right\}\\
  \nonumber
  &\hspace{2.5cm}- O\!\left(\#\left\{ E \in \cF(H) \,:\, \q < \frac{|\Delta_E|}{N_E},\,\, \lambda_0 < \frac{|\Delta_E|}{H}\right\}\right).
\end{align}
By \cref{FN_bigm},
$$\#\left\{ E \in \cF(H) \,:\, \q < \frac{|\Delta_E|}{N_E} < H^{\frac{1}{6} - \eps},\,\, \lambda_0 < \frac{|\Delta_E|}{H} \right\} \ll \q^{-1+\eps}\#\cF(H).$$
By \cref{N_56_bound},
$$\#\left\{ E \in \cF(H) \,:\, H^{\frac{1}{6} - \eps} < \frac{|\Delta_E|}{N_E} ,\,\, \lambda_0 < \frac{|\Delta_E|}{H} \right\} \ll H^{\frac{19}{24} + \eps}.$$

Comparing \cref{rho_def} to \cref{table:redp,table:red2_all,table:red3_all}, we see that
\begin{align}
  \frac{\#S_{\QD,\m}}{\QD^2} = \prod_{\phantom{{}_*}p \,<\, \q} \rho(p,\m).
  \label{eq:SQm_prod}
\end{align}
Note that for the family $\cFY(H)$, one acquires an extra factor of $\tfrac{1}{36}$ on the right, coming from the requirement that $a = r \ \mathrm{mod}\ 6$ and $b = t \ \mathrm{mod}\ 6$.

From the prime number theorem \cite[(6.12)]{MV},
$$\QD \asymp \exp\!\left(3\q + O\!\left(\frac{\q}{\exp(c\sqrt{\log \q})}\right)\right)$$
for some $c > 0$ as $\q \to \infty$. By \cref{d4d6_by_type} we may apply \cref{disc_dist} to \eqref{eq:SQm_decomp}, deducing
\begin{align*}
  &\#\left\{ E \in \cF(H) \,:\, \lambda_0 < \frac{N_E}{H} < \lambda_1 \right\}\\
  &\hspace{2cm}= \sum_{\substack{-\frac{496}{\lambda_0} \,<\, \m \,<\, \frac{64}{\lambda_0} \\ \m \,\neq\, 0}} \left(\frac{4H^{\frac{5}{6}}}{\zeta^{(\QD)}(10)} \big|F_\Delta(\m\lambda_1) - F_\Delta(\m\lambda_0)\big|\prod_{\phantom{{}_*}p \,<\, \q} \rho(p,\m) + O\!\left(H^{\half+\frac{3q}{\log H}} \right)\right)\\
  &\hspace{2.5cm}+ O\!\left(\q^{-1+\eps} \#\cF(H) + H^{\frac{19}{24} + \eps}\right)
\end{align*}
(the coefficient $4$ becomes $\tfrac{1}{9}$ for $\cFY$). \Cref{FH_size} and some elementary manipulations then yield \eqref{eq:conductor_distribution}.
\end{proof}

\begin{proof}[\hypertarget{proof:thm2}{Proof of \cref{thm:small_conductors}}]
  We prove \cref{thm:small_conductors} for the family $\cF$ only. The proof for $\cFY$ is very similar; we point out the most important differences as we go.
  
  Let $\QD$, $C$, and $S_{\QD,\m}$ be as as in \cref{def:QC}. These are functions of $\q$.
  Observe that, by construction,
  \begin{align*}
    \sum_{\m\mid C} \# S_{\QD,\m} &= \#\{ (a,b) \in (\Z/\QD\Z)^2 \,: \text{$(a,b)$ satisfies 
      \eqref{eq:d4d6_minimal}} \}
  \end{align*}
  (both \eqref{eq:d4d6_minimal} and \eqref{eq:d4d6_ab} for the family $\cFY$), thus
  \begin{align}
    \lim_{\q\to\infty} \frac{1}{\QD^2}\sum_{\m\mid C} \# S_{\QD,\m} = \prod_p 1 - p^{-10}\label{SQ_sum_size}
  \end{align}
  (with a factor of $\tfrac{1}{36}$ on the right for the family $\cFY$, and the product over $p \geqslant 5$).
  
  Combining \eqref{SQ_sum_size} and \eqref{eq:SQm_prod},
  \begin{align}
    \sum_{\m=1}^\infty \prod_p \rho(p,\m) = \frac{1}{\zeta(10)}
    \label{rho_sum}
  \end{align}
  ($\zeta^{(6)}(10)^{-1}$ for $\cFY$).  
  We can now compute
  \begin{align}
    &\frac{\zeta(10)}{\zeta(2)} \sum_{\m = 1}^\infty \big(F_\Delta(\infty) - F_\Delta(\m\lambda) + F_\Delta(-\m\lambda) - F_\Delta(-\infty)\big)\prod_p \frac{\rho(p,\m)}{1 - p^{-2}}&& \nonumber\\
    &\hspace{2cm}= \zeta(10) \sum_{\m = 1}^\infty \big(1 - (F_\Delta(\m\lambda) - F_\Delta(-\m\lambda))\big)\prod_p \rho(p,\m)&& \nonumber\\
    &\hspace{2cm}= 1 - \zeta(10) \sum_{\m = 1}^\infty \big(F_\Delta(\m\lambda) - F_\Delta(-\m\lambda)\big)\prod_p \rho(p,\m) &&\text{(using \eqref{rho_sum})} \nonumber\\
    &\hspace{2cm}= 1 - O(\lambda^{\frac{5}{6}}) &&\text{(using \cref{FN_small_bound})}. \nonumber
  \end{align}
  Using \cref{thm:conductor_distribution_FN},
  \begin{align*}
    &\frac{\#\!\left\{ E \in \cF(H) \,:\, N_E < \lambda H \right\}}{\#\cF(H)}\\
    &\hspace{1cm} = 1 - \frac{\#\!\left\{ E \in \cF(H) \,:\, N_E \geqslant \lambda H \right\}}{\#\cF(H)}\\
    &\hspace{1cm} = 1 - \frac{\zeta(10)}{\zeta(2)} \sum_{\m = 1}^\infty \big(F_\Delta(\infty) - F_\Delta(\m\lambda) + F_\Delta(-\m\lambda) - F_\Delta(-\infty)\big)\prod_p \frac{\rho(p,\m)}{1 - p^{-2}}\\
    &\hspace{1.5cm}+ O( (\log H)^{-1+\eps})\\
    &\hspace{1cm} = 1 - (1 - O(\lambda^{\frac{5}{6}})) + O( (\log H)^{-1+\eps}) \\
    &\hspace{1cm} \ll \lambda^{\frac{5}{6}}. \hspace{7.75cm}\text{(because $\lambda \gg (\log H)^{-1}$).}
  \end{align*}
  
   Using \cref{FH_size} and the fact that $\lambda \gg (\log H)^{-1}$,
  \begin{align*}
    \#\{E \in \cF(H) \,:\, N_E < \lambda H\} \geqslant \#\cF(\lambda H) \gg \lambda^{\frac{5}{6}} \#\cF(H),
  \end{align*}
  proving the theorem's lower bound.
\end{proof}

\Cref{fig:thm1}'s plots were computed numerically using the identity
\begin{align}
  \label{eq:fig_identity}
  \begin{aligned}
  &\frac{\zeta(10)}{\zeta(2)} \sum_{\m = 1}^\infty \big(F_\Delta(\m\lambda_1) - F_\Delta(-\m\lambda_1)\big)\prod_p \frac{\rho(p,\m)}{1 - p^{-2}}
  \\&
  \hspace{5cm}
  = 1 + \frac{\zeta(10)}{\zeta(2)} \sum_{1 \leqslant \m \leqslant \frac{496}{\lambda_0}} \big(F_\Delta(\m\lambda_1) - F_\Delta(-\m\lambda_1) - 1\big)\prod_p \frac{\rho(p,\m)}{1 - p^{-2}}
  ,
  \end{aligned}
\end{align}
valid for any $\lambda_0 \leqslant \lambda_1$ (choosing $\lambda_0 \approx \lambda_1$ is most convenient). This identity follows from \eqref{rho_sum}.

\begin{proof}[Proof of \cref{thm:tail_estimate}]
  Once again, we prove the theorem for the family $\cF(H)$ only, as the proof for $\cFY(H)$ is virtually identical.
  
  For $0 < \delta < \tfrac{1}{6}$ and $\lambda > 0$, define
  \begin{align*}
    &\cF(H)_1 \coloneqq \!\left\{E \in \cF(H) \,:\, N_E < \lambda H,\,\, \frac{|\Delta_E|}{N_E} < \M\right\},\\
    &\cF(H)_2 \coloneqq \!\left\{E \in \cF(H) \,:\, H^{1 - \delta} < N_E < \lambda H,\,\, \frac{|\Delta_E|}{N_E} > \M\right\},\\
    &\cF(H)_3 \coloneqq \!\left\{E \in \cF(H) \,:\, N_E < H^{1 - \delta},\,\, \frac{|\Delta_E|}{N_E} > \M\right\},
  \end{align*}
  where $\M \in \R_{>0}$ will be chosen later.

  Consider $\cF(H)_1$.
  \begin{align*}
    &\#\cF(H)_1 \leqslant \#\!\left\{E \in \cF(H) \,:\, |\Delta_E| < \lambda \M H\right\}.
  \end{align*}
  We apply \cref{disc_dist} with $\QM = 1$, (for the family $\cFY$, instead $\QM = 6$ and $S_\QM$ corresponding to only the condition that $E \in \cFY(H)$), yielding
  \begin{align}\label{F1_bound}
    \#\cF(H)_1 &\ll H^{\frac{5}{6}} (F_\Delta(\lambda \M) - F_\Delta(-\lambda \M)) + H^\half
    \ll H^{\frac{5}{6}} (\lambda \M)^{\frac{5}{6}} + H^\half.
  \end{align}
  The second bound above employs \cref{FDelta_bound}.

  Consider $\cF(H)_2$. By \cref{FN_bigm} and \cref{FH_size},
  \begin{align}\label{F2_bound}
    \#\cF(H)_2 \ll  H^{\frac{5}{6}} \M^{-1+\eps}.
  \end{align}
  As $\delta$ does not feature in \eqref{F2_bound}, we take $\delta = \tfrac{1}{6} - \eps$.

  Consider $\cF(H)_3$. Via the same reasoning as was used in the proof of \cref{N_56_bound},
  \begin{align}
    \#\cF(H)_3
    &\ll \sum_{\substack{E \,\in\, \cF(H)^{\mathrm{min}} \\ N_E \,<\, H^{\frac{5}{6}}\M^{-\frac{2}{7}}}} \left(\frac{H}{N_E}\right)^{\frac{1}{6}} 
    + \sum_{\substack{E \,\in\, \cF(H)^{\mathrm{min}} \\ N_E \,>\, H^{\frac{5}{6}}\M^{-\frac{2}{7}} \\ \frac{|\Delta_E|}{N_E} \,>\, \M N_E^2H^{-\frac{5}{3}}}} \left(\frac{H}{\Delta_E}\right)^{\frac{1}{6}} \nonumber\\
    &\ll H^{\frac{31}{36} + \eps} \M^{-\frac{5}{21}}. \label{F3_bound}
  \end{align}

  Combining \eqref{F1_bound}, \eqref{F2_bound}, and \eqref{F3_bound},
  \begin{align*}
    &\#\!\left\{E \in \cF(H) \,:\, N_E < \lambda H\right\}
    \ll \max\!\left\{ (H\lambda \M)^{\frac{5}{6}} + H^{\frac{1}{2}},\,\, \M^{-1+\eps},\,\, H^{\frac{31}{36} + \eps} \M^{-\frac{5}{21}} \right\}.
  \end{align*}
  Defining $\x \coloneqq \lambda H$ and setting
  \begin{align*}
    \M = H^{\frac{7}{270}}\left(\frac{H}{\x}\right)^{\!\frac{7}{9}}
  \end{align*}
  yields the upper bound of \cref{thm:tail_estimate}. The lower bound follows from comparison with $\#\cF(\x)$, like in the \hyperlink{proof:thm2}{proof of \cref*{thm:small_conductors}}.
\end{proof}

\renewcommand{\bibliofont}{\normalfont\small} 
\bibliographystyle{amsalpha}
\bibliography{murmurationsbib}{}

\end{document}